\documentclass[10pt]{article}
\usepackage[english]{babel}
\usepackage{latexsym}
\usepackage{amsmath}
\usepackage{amssymb}
\usepackage[latin1]{inputenc}
\usepackage{enumerate}
\usepackage{graphicx}

\usepackage[sc]{mathpazo}

\usepackage{bbm}
\usepackage{mathtools}
\usepackage[amsthm,amsmath,thmmarks,thref]{ntheorem}
\theoremseparator{.}

\newtheorem{definicion}{Definition}[section]
\newtheorem{teorema}[definicion]{Theorem}
\newtheorem{lema}[definicion]{Lemma}

\newtheorem{claim}[definicion]{Claim}
\newtheorem{corolario}[definicion]{Corollary}
\newtheorem{fact}[definicion]{Fact}

\title{On the reducibility of isomorphism relations}
\author{Tapani Hyttinen, Miguel Moreno\\ University of Helsinki}
\begin{document}
\maketitle
\begin{abstract}
We study the Borel reducibility of isomorphism relations in the generalized Baire space $\kappa^\kappa$. In the main result we show for inaccessible $\kappa$, that if $T$ is a classifiable theory and $T'$ is stable with OCP, then the isomorphism of models of $T$ is Borel reducible to the isomorphism of models of $T'$.
\end{abstract}
\section{Introduction}
One of the main motivations behind writing
\cite{FHK13} was the possibility that Borel reducibility
in generalized Baire spaces can be used to measure
the complexity of countable first-order theories
(we concentrated on elementary classes with countable vocabulary, since for them
there is a lot of structure theory):
We say that $T$ is simpler than $T'$
if the isomorphism relation among models of
$T$ with universe $\kappa$ ($\cong_{T}$) is Borel (or continuously)
reducible to the isomorphism relation among
the models of $T'$ with universe $\kappa$. Here, and throughout the paper, we assume that $\kappa^{<\kappa}=\kappa>\aleph_0$ (see \cite{FHK13} for the discussion why $\kappa=\aleph_0$ does not work). The results reviewed in this introduction often require further assumptions on $\kappa$, but for sake of clarity the details are omitted and the reader is referred to the original papers for the exact assumptions on $\kappa$. The question when
such reduction exists
turned out to be harder
than we expected.\\ \\
In \cite{FHK13} the results were negative: It was shown that if $T$
is classifiable (superstable with NOTOP and NDOP) and shallow
and $T'$ is not, then $\cong_{T'}$ is not
Borel reducible to $\cong_{T}$ and that at least consistently,
if $T$ is classifiable and $T'$ is not, then
$\cong_{T'}$ is not Borel reducible to $\cong_{T}$.\\ \\
In \cite{HK}, some positive results were obtained: If $V=L$, then all the $\Sigma_1^1$ equivalence relations are reducible to $\cong_{DLO}$,
where $DLO$ is the theory of dense linear orderings without
end points (in \cite{FS} it was proved for $\kappa=\omega$ that $\cong_{DLO}$ is Borel complete, and the proof in \cite{HK} is similar). Also it was shown that consistently the same is true
for $T_{\omega+\omega}$ (see below). Obviously, there are theories for
which this holds trivially e.g. graphs (even random graphs with a bit more work).
Also it was shown that if a theory $T'$ has this property and $V=L$,
then $\cong_{T'}$ is $\Sigma^{1}_{1}$-complete.\\ \\
Finally, by combining Corollary 15 from \cite{FHK14} and
the proof of Theorem 16 from \cite{HK}, it follows
that if $T$ is classifiable and shallow, then
$\cong_{T}$ is reducible to $\cong_{T_{\omega +\omega}}$.\\ \\
In this paper we improve two of the results mentioned
above. We start by showing that if $T$ is classifiable,
then $\cong_{T}$ is Borel reducible to $\cong_{T_{\omega}}$
(again, see below) and that if $V=L$, then $\cong_{T_{\omega}}$
is $\Sigma^{1}_{1}$-complete. And then
under heavy assumptions on $\kappa$, we generalize a lot: We show that
if $T$ is classifiable and $T'$ is stable with OCP (see below),
then $\cong_{T}$ is continuously reducible to $\cong_{T'}$
and if in addition $V=L$, then $\cong_{T'}$ is $\Sigma^{1}_{1}$-complete.
The property OCP implies that $T'$ is unsuperstable
and it is common among stable unsuperstable theories.
E.g. both $T_{\omega}$ and $T_{\omega +\omega}$ mentioned above have it.
It is also easy to find complete theories of abelian groups (or more generally
elementary theories of ultrametric spaces)
that have the property. What does not seem to be easy, is to find
a strictly stable theory that does not have the property.\\ \\
We are going to work on the generalised Baire space $\kappa^\kappa$  with the following topology. For every $\zeta\in \kappa^{<\kappa}$, we call the set $$[\zeta]=\{\eta\in\kappa^\kappa|\zeta\subset\eta\}$$ a basic open set. Then the open sets are of the form $\bigcup X$ where $X$ is a collection of basic open sets. The $\kappa$-Borel space of $\kappa^\kappa$ is the smallest set, which contains the basic open sets, and is closed under unions and intersections, both of length $\kappa$. A Borel set, is any element of the $\kappa$-Borel space. Suppose $X$ and $Y$ are subsets of $\kappa^\kappa$, a function $f:X\rightarrow Y$ is a Borel function, if for every open set $A\subseteq Y$, $f^{-1}[A]$ is a Borel set in $X$ of $\kappa^\kappa$.\\ \\
Suppose $X$ and $Y$ are subsets of $\kappa^\kappa$, let $E_1$ and $E_2$ be equivalent relations on $X$ and $Y$ respectively. If a function $f:X\rightarrow Y$ satisfies $E_1(x,y)\Leftrightarrow E_2(f(x),f(y))$, we say that $f$ is a reduction of $E_1$ to $E_2$. If there exists a Borel function that is a reduction, we say that $E_1$ is Borel reducible to $E_2$ and we denote it by $E_1\leq_B E_2$. If there exists a continuous function that is a reduction, we say that $E_1$ is continuously reducible to $E_2$ and we denote it by $E_1\leq_c E_2$.\\ \\
For every regular cardinal $\mu<\kappa$, we say that a set $A\subseteq \kappa$ is a $\mu$-club if it is unbounded and closed under $\mu$-limits. Clearly the intersection of two $\mu$-clubs is also a $\mu$-club and every $\mu$-club is stationary.\\
On the space $\kappa ^\kappa$, we say that $f,g\in \kappa^\kappa$ are $E^\kappa_{\mu\text{-club}}$ equivalent ($f\  E^\kappa_{\mu\text{-club}}\ g$) if the set $\{\alpha<\kappa|f(\alpha)=g(\alpha)\}$ contains a $\mu$-club.


\section{Classifiable Theories}
Let us fix a countable relation vocabulary $\mathcal{L}=\{R_{(n,m)}|n,m\in \omega\backslash\{0\}\}$, where $R_{(n,m)}$ is an $n$-ary relation. Fix a bijection $g:\omega\backslash\{0\}\times \omega\backslash\{0\}\rightarrow\omega$, define $P_{g(n,m)}:=R_{(n,m)}$ and rewrite $\mathcal{L}=\{P_n|n<\omega\}$. Denote $g^{-1}(\alpha)$ by $(g^{-1}_1(\alpha),g^{-1}_2(\alpha))$. When we describe a complete theory $T$ in a vocabulary $L\subseteq \mathcal{L}$, we think it as a complete $\mathcal{L}$-theory $T\cup \{\forall \bar{x}\neg
P_n(\bar{x})|P_n\in \mathcal{L}\backslash L\}$. We can code $\mathcal{L}$-structures with domain $\kappa$ as follows.
\begin{definicion}
Fix $\pi$ a bijection between $\kappa^{<\omega}$ and $\kappa$. For every $\eta\in \kappa^\kappa$ define the structure $\mathcal{A}_\eta$ with domain $\kappa$ as follows.\\
For every tuple $(a_1,a_2,\ldots , a_n)$ in $\kappa^n$ $$(a_1,a_2,\ldots , a_n)\in P_m^{\mathcal{A}_\eta}\Leftrightarrow n=g_1^{-1}(m) \text{ and }\eta(\pi(m,a_1,a_2,\ldots,a_n))>0.$$
\end{definicion}
This defines a map from $\kappa^\kappa$ onto the set of $\mathcal{L}$-structures with domain $\kappa$.
\begin{definicion}
(The isomorphism relation) Assume $T$ a complete first order theory in a countable vocabulary and $\eta,\xi\in \kappa^\kappa$, we define $\cong_T$ as the relation $$\{(\eta,\xi)|(\mathcal{A}_\eta\models T, \mathcal{A}_\xi\models T, \mathcal{A}_\eta\cong \mathcal{A}_\xi)\text{ or } (\mathcal{A}_\eta\not\models T, \mathcal{A}_\xi\not\models T)\}$$
\end{definicion}
The following game is the usual Ehrenfeucht-Fra\"{i}ss\'e game with structures of domain $\kappa$ and moves coded by ordinals.
The Ehrenfeucht-Fra\"{i}ss\'e game will be useful for the study of $\cong_T$ when $T$ is classifiable.\\
Shelah proved \cite{Sh} that when $T$ is classifiable, two models $\mathcal{A}$ and $\mathcal{B}$ are isomorphic if and only if the second player has a winning strategy in the Ehrenfeucht-Fra\"{i}ss\'e game EF$^\kappa_\omega (\mathcal{A},\mathcal{B})$.\\
We will show that the existence of a winning strategy depends on the existence of a club on $\kappa$. We can study the isomorphism relation by studying the relation $E^\kappa_{\mu\text{-club}}$.
\begin{definicion}
(Ehrenfeucht-Fra\"{i}ss\'e game) Fix $\{X_\gamma\}_{\gamma<\kappa}$ an enumeration of the elements of $\mathcal{P}_\kappa(\kappa)$ and $\{f_\gamma\}_{\gamma<\kappa}$ an enumeration of all the functions with domain in $\mathcal{P}_\kappa(\kappa)$ and range in $\mathcal{P}_\kappa(\kappa)$.
For every pair of structures $\mathcal{A}$ and $\mathcal{B}$ with domain $\kappa$, the EF$^\kappa_\omega (\mathcal{A},\mathcal{B})$ is a game played by the players $\bf{I}$ and $\bf{II}$ as follows.\\
In the $n$-th move, first $\bf{I}$ chooses an ordinal $\beta_n<\kappa$ such that $X_{\beta_{n-1}}\subset X_{\beta_n}$, and then $\bf{II}$ an ordinal $\theta_n<\kappa$ such that $X_{\beta_n}\subseteq dom(f_{\theta_n})\cap rang(f_{\theta_n})$ and $f_{\theta_{n-1}}\subset f_{\theta_n}$ (if $n=0$ then $X_{\beta_{n-1}}=\emptyset$ and $f_{\theta_{n-1}}=\emptyset$).\\
The game finishes after $\omega$ moves. The player $\bf{II}$ wins if $\cup_{i<\omega}f_{\theta_i}:A\rightarrow B$ is a partial isomorphism, otherwise the player $\bf{I}$ wins.
\end{definicion}
For every $\alpha<\kappa$ we can define the restricted game EF$^\kappa_\omega (\mathcal{A}\restriction_\alpha,\mathcal{B}\restriction_\alpha)$ for structures $\mathcal{A}$ and $\mathcal{B}$ with domain $\kappa$, as follows.\\
In the $n$-th move, first $\bf{I}$ choose an ordinal $\beta_n<\alpha$ such that $X_{\beta_n}\subset \alpha$, $X_{\beta_{n-1}}\subseteq X_{\beta_n}$, and then $\bf{II}$ an ordinal $\theta_n<\alpha$ such that $dom(f_{\theta_n}),rang(f_{\theta_n})\subset \alpha$, $X_{\beta_n}\subseteq dom(f_{\theta_n})\cap rang(f_{\theta_n})$ and $f_{\theta_{n-1}}\subseteq f_{\theta_n}$ (if $n=0$ then $X_{\beta_{n-1}}=\emptyset$ and $f_{\theta_{n-1}}=\emptyset$). The game finishes after $\omega$ moves. The player $\bf{II}$ wins if $\cup_{i<\omega}f_{\theta_i}:A\restriction_\alpha\rightarrow B\restriction_\alpha$ is a partial isomorphism, otherwise the player $\bf{I}$ wins.\\
Notice that now wining strategies are functions from $\kappa^{<\kappa}$ to $\kappa$.\\
We will write $\bf{I}\uparrow$ EF$^\kappa_\omega (\mathcal{A}\restriction_\alpha,\mathcal{B}\restriction_\alpha)$ when $\bf{I}$ has a winning strategy in the game EF$^\kappa_\omega (\mathcal{A}\restriction_\alpha,\mathcal{B}\restriction_\alpha)$, similarly we write $\bf{II}\uparrow$ EF$^\kappa_\omega (\mathcal{A}\restriction_\alpha,\mathcal{B}\restriction_\alpha)$ when $\bf{II}$ has a winning strategy.  
\begin{lema}
If $\mathcal{A}$ and $\mathcal{B}$ are structures with domain $\kappa$, then the following hold:
\begin{itemize}
\item $\bf{II}\uparrow$ EF$^\kappa_\omega (\mathcal{A},\mathcal{B})\Longleftrightarrow \bf{II}\uparrow$ EF$^\kappa_\omega (\mathcal{A}\restriction_\alpha,\mathcal{B}\restriction_\alpha)$ for club-many $\alpha$.
\item $\bf{I}\uparrow$ EF$^\kappa_\omega (\mathcal{A},\mathcal{B})\Longleftrightarrow \bf{I}\uparrow$ EF$^\kappa_\omega (\mathcal{A}\restriction_\alpha,\mathcal{B}\restriction_\alpha)$ for club-many $\alpha$.
\end{itemize}
\end{lema}
\begin{proof}
Let us start by, $\bf{II}\uparrow$ EF$^\kappa_\omega (\mathcal{A},\mathcal{B})\Rightarrow \bf{II}\uparrow$ EF$^\kappa_\omega (\mathcal{A}\restriction_\alpha,\mathcal{B}\restriction_\alpha)$ for club-many $\alpha$.\\
Suppose $\sigma$ is a winning strategy for $\bf{II}$ and denote by $C_\sigma$ the club $\{\alpha< \kappa: \sigma[\alpha^{<\omega}]\subseteq \alpha\}$. Define the function $H:\kappa\rightarrow \kappa$ by $H(\alpha)=sup(rang(f_\alpha)\cup dom(f_\alpha)\cup X_\alpha)$, this function defines the club $C_H:=\{\gamma<\kappa|\forall\alpha<\gamma(H(\alpha)<\gamma)\}$.\\
For all $\alpha\in C_\sigma\cap C_H$, $\alpha$ satisfies $\sigma[\alpha^{<\omega}]\subseteq \alpha$ and every $\beta<\alpha$ satisfies $sup(rang(f_\beta)\cup dom(f_\beta))<\alpha$. Then the domain and range of $f_\beta$ are subsets of $\alpha$. We conclude that $\sigma\restriction_{\alpha^{<\omega}}$ is a winning strategy for $\bf{II}$ in the restricted game EF$^\kappa_\omega (\mathcal{A}\restriction_\alpha,\mathcal{B}\restriction_\alpha)$. Since the intersection of clubs is a club, then there are club many $\alpha$ such that $\bf{II}\uparrow$ EF$^\kappa_\omega (\mathcal{A}\restriction_\alpha,\mathcal{B}\restriction_\alpha)$.\\ \\
The case $\bf{I}\uparrow$ EF$^\kappa_\omega (\mathcal{A},\mathcal{B})\Rightarrow \bf{I}\uparrow$ EF$^\kappa_\omega (\mathcal{A}\restriction_\alpha,\mathcal{B}\restriction_\alpha)$ for club-many $\alpha$ is similar.\\ \\
The two directions (from left to right) are proved in the same way, and thus we show only one. Suppose there are club many $\alpha$ such that $\bf{II}\uparrow$ EF$^\kappa_\omega (\mathcal{A}\restriction_\alpha,\mathcal{B}\restriction_\alpha)$ (denote this club by $C_{\bf{II}}$) and there is no winning strategy for $\bf{II}$ in the game EF$^\kappa_\omega (\mathcal{A},\mathcal{B})$. Since this game is a determined game, then $\bf{I}\uparrow$ EF$^\kappa_\omega (\mathcal{A},\mathcal{B})$. We already showed that this implies the existence of club many $\alpha$ such that $\bf{I}\uparrow$ EF$^\kappa_\omega (\mathcal{A}\restriction_\alpha,\mathcal{B}\restriction_\alpha)$ (denote this club by $C_{\bf{I}}$). Since the intersection of clubs is a club, then $C_{\bf{I}}\cap C_{\bf{II}}\neq \emptyset$. Therefore, there exists $\alpha$, such that both players have a winning strategy for the game EF$^\kappa_\omega (\mathcal{A}\restriction_\alpha,\mathcal{B}\restriction_\alpha)$, a contradiction.
\end{proof}
\begin{corolario}
For every $\mu<\kappa$ and every pair of structures $\mathcal{A}$ and $\mathcal{B}$ with domain $\kappa$, 
\begin{equation*}
{\bf II}\uparrow\text{ EF} ^\kappa_\omega (\mathcal{A},\mathcal{B})\Longleftrightarrow {\bf II}\uparrow\text{ EF} ^\kappa_\omega (\mathcal{A}\restriction_\alpha,\mathcal{B}\restriction_\alpha)\text{ for }\mu\text{-club-many } \alpha
\end{equation*}
\end{corolario}
By Shelah's result \cite{Sh} we know that a classifiable theory, $\bf{II}\uparrow$ EF$^\kappa_\omega (\mathcal{A},\mathcal{B})$ implies that $\mathcal{A}$ and $\mathcal{B}$ are isomorphic. Therefore, for all $\eta,\xi\in \kappa^\kappa$, the player $\bf{II}\uparrow$ EF$^\kappa_\omega (\mathcal{A}_\eta,\mathcal{A}_\xi)$ if and only if $\eta\cong_T\xi$. We can use the restricted games to define new relations, one relation for each $\alpha$. By the previous corollary, we can use these relations to study the isomorphism relation.
\begin{definicion}
Assume $T$ is a complete first order theory in a countable vocabulary. For every $\alpha< \kappa$ and $\eta,\xi\in \kappa^\kappa$, we write $\eta\ R_{EF}^\alpha\ \xi$ if one of the following holds, $\mathcal{A}_\eta\restriction_\alpha\not\models T$ and $\mathcal{A}_\xi\restriction_\alpha\not\models T$, or $\mathcal{A}_\eta\restriction_\alpha\models T$, $\mathcal{A}_\xi\restriction_\alpha\models T$ and $\bf{II}\uparrow$ EF$^\kappa_\omega (\mathcal{A}_\eta\restriction_\alpha,\mathcal{A}_\xi\restriction_\alpha)$.
\end{definicion}
Notice that for each $\alpha\leq \kappa$, $R_{EF}^\alpha$ is a relation on $\kappa^\kappa\times\kappa^\kappa$ but it is not necessarily an equivalence relation. Fortunately there are club-many $\alpha$ such that $R_{EF}^\alpha$ is an equivalence relation. 
\begin{lema}
For every complete first order theory $T$ in a countable vocabulary, there are club many $\alpha$ such that $R_{EF}^\alpha$ is an equivalence relation. 
\end{lema}
\begin{proof}
Define the following functions:
\begin{itemize}
\item $h_1:\kappa\rightarrow \kappa$, $h_1(\alpha)=\gamma$ where $f_\gamma$ is the identity function of $X_\alpha$.
\item $h_2:\kappa\rightarrow \kappa$, $h_2(\alpha)=\gamma$ where $f_\alpha^{-1}=f_\gamma$.
\item $h_3:\kappa^2\rightarrow \kappa$, $h_3(\alpha,\beta)=X_{\alpha}\cup X_{\beta}= X_\gamma$.
\item $h_4:\kappa\rightarrow \kappa$, $h_4(\alpha)=rang(f_\alpha)= X_\gamma$.
\item $h_5:\kappa\rightarrow \kappa$, $h_5(\alpha)=dom(f_\alpha)= X_\gamma$.
\item $h_6:\kappa^2\rightarrow \kappa$, $h_6(\alpha,\beta)=\gamma$ where $f_{\alpha}\circ f_{\beta}=f_\gamma$, $f_{\alpha}\circ f_{\beta}$ is defined on the set $f_{\beta}^{-1}[rang(f_{\beta})\cap dom (f_{\alpha})]$.
\end{itemize}
Each of these functions defines a club,
\begin{itemize}
\item $C_i=\{\gamma<\kappa|\forall\alpha<\gamma(h_i(\alpha)<\gamma)\}$ for $i\in\{1,2,4,5\}$.
\item $C_i=\{\gamma<\kappa|\forall\beta,\alpha<\gamma(h_i(\alpha,\beta)<\gamma)\}$ for $i\in\{3,6\}$.
\end{itemize}
Denote by $C$ the club $\cap_{i=1}^6C_i$. We will show that for every $\alpha\in C$, $R_{EF}^\alpha$ is an equivalence relation.\\ \\
By definition $\eta\ R_{EF}^\alpha\ \xi$ implies that either both $\mathcal{A}_\eta$ and $\mathcal{A}_\xi$ are models of $T$ or non of them is a model of $T$. Thus  $R_{EF}^\alpha=R^-\cup R^+$, where $R^-$ is the restriction of $R_{EF}^\alpha$ to the set $A=\{\eta\in \kappa|\mathcal{A}_\eta\not\models T\}$ and $R^+$ is the restriction of $R_{EF}^\alpha$ to the complement of $A$. Since $R^-\cap R^+=\emptyset$, it is enough to prove that $R^-$ and $R^+$ are equivalence relations.\\
By definition it is easy to see that $R^-=A\times A$, therefore $R^-$ is an equivalence relation.
Now we will prove that $R^+$ is an equivalence relation.\\ \\
{\bf Reflexivity}\\
By the way $C_1$ was defined, for every $\beta<\alpha$, $h_1(\beta)<\alpha$ and $f_{h_1(\beta)}$ is the identity function of $X_\beta$. Therefore, the function $\sigma((\beta_0,\beta_1,\ldots ,\beta_n))=h_1(\beta_n)$ is a winning strategy for $\bf{II}$ in the game EF$^\kappa_\omega (\mathcal{A}_\eta\restriction_\alpha,\mathcal{A}_\eta\restriction_\alpha)$.\\ \\
{\bf Symmetry}\\
Let $\sigma$ be a winning strategy for $\bf{II}$ in the game EF$^\kappa_\omega (\mathcal{A}_\eta\restriction_\alpha,\mathcal{A}_\xi\restriction_\alpha)$. Since $\alpha\in C_2$ and $\sigma((\beta_0,\beta_1,\ldots,\beta_n))<\alpha$, we know that $h_2(\sigma((\beta_0,\beta_1,\ldots,\beta_n)))<\alpha$. Notice that if $\cup_{i<\omega}f_{\theta_i}:\alpha\rightarrow \alpha$ is a partial isomorphism from $\mathcal{A}_\eta\restriction_\alpha$ to $\mathcal{A}_\xi\restriction_\alpha$, then $\cup_{i<\omega}f_{h_2(\theta_i)}=\cup_{i<\omega}f_{\theta_i}^{-1}$ is a partial isomorphism from $\mathcal{A}_\xi\restriction_\alpha$ to $\mathcal{A}_\eta\restriction_\alpha$. Therefore, the function $\sigma'((\beta_0,\beta_1,\ldots,\beta_n))=h_2(\sigma((\beta_0,\beta_1,\ldots,\beta_n)))$ is a winning strategy for $\bf{II}$ in the game EF$^\kappa_\omega (\mathcal{A}_\xi\restriction_\alpha,\mathcal{A}_\eta\restriction_\alpha)$.\\ \\
{\bf Transitivity}\\
Let $\sigma_1$ and $\sigma_2$ be two winning strategies for $\bf{II}$ on the games EF$^\kappa_\omega (\mathcal{A}_\eta\restriction_\alpha,\mathcal{A}_\xi\restriction_\alpha)$ and EF$^\kappa_\omega (\mathcal{A}_\xi\restriction_\alpha,\mathcal{A}_\zeta\restriction_\alpha)$, respectively.\\
For a given tuple $(\beta_0,\beta_1,\ldots,\beta_n)$ let us construct by induction the tuples $(\gamma_0,\gamma_1,\ldots,\gamma_n)$, $(\beta'_0,\beta'_1,\ldots,\beta'_{2n},\beta'_{2n+1})$, and the functions $f_{(1,n)}$, $g_n$ and $f_{(2,n)}$:
\begin{enumerate}
\item Let $\beta'_0=\beta_0$ and for $i>0$, let $\beta'_{2i}$ be the least ordinal such that $X_{\beta'_{2i-1}}\cup X_{\beta_i}= X_{\beta'_{2i}}$.
\item $f_{(1,i)}:=f_{\sigma_1((\beta'_0,\beta'_1,\ldots,\beta'_{2i-1},\beta'_{2i}))}$.
\item $\gamma_i$ is the ordinal such that $X_{\gamma_i}=rang (f_{(1,i)})$.
\item $g_i:=f_{\sigma_2((\gamma_0,\gamma_1,\ldots,\gamma_i))}$.
\item $\beta'_{2i+1}$ is the ordinal such that $X_{\beta'_{2i+1}}=dom (g_i)$.
\item $f_{(2,i)}:=f_{\sigma_1((\beta'_0,\beta'_1,\ldots,\beta'_{2i},\beta'_{2i+1}))}$.
\end{enumerate}
Define the function $\sigma:\alpha^{<\omega}\rightarrow \alpha$ by $\sigma((\beta_0,\beta_1,\ldots,\beta_n))=\theta_{n}$, where $\theta_n$ is the ordinal such that $f_{\theta_n}=g_n\circ (f_{(2,n)}\restriction_{f_{(2,n)}^{-1}[dom (g_n)]})$. It is easy to check that for every $n$, the tuples $(\gamma_0,\gamma_1,\ldots,\gamma_n)$ and $(\beta'_0,\beta'_1,\ldots,\beta'_{2n+1})$ are elements of $\alpha^{<\omega}$, and the functions $f_{(1,n)}$, $g_n$, $f_{(2,n)}$ and $f_{\theta_n}$ are well defined; it is also easy to check that $\sigma((\beta_0,\beta_1,\ldots,\beta_n))$ is a valid move.\\
Let us show that $\cup_{n<\omega}f_{\theta_n}$ is a partial isomorphism. It is clear that $rang(f_{(2,n)})\subseteq rang(f_{(1,n+1)})$. By 3 and 4 in the induction, we can conclude that $rang(f_{(2,n)})$ is a subset of $dom(g_{n+1})$. Then 
$rang (\cup_{n<\omega}(f_{(2,n)}))\subseteq dom (\cup_{n<\omega}(g_n))$, so $$\cup_{n<\omega}(g_n \circ (f_{(2,n)}\restriction_{f_{(2,n)}^{-1}[dom(g_n)]}))=\cup_{n<\omega}(g_n) \circ \cup_{n<\omega}(f_{(2,n)}).$$ Since $\sigma_1$ and $\sigma_2$ are winning strategies, we know that $\cup_{n<\omega}(g_n)$ and $\cup_{n<\omega}(f_{(2,n)})$ are partial isomorphism. Therefore $\cup_{n<\omega}f_{\theta_n}$ is a partial isomorphism and $\sigma$ is a winning strategy for $\bf{II}$ on the game EF$^\kappa_\omega (\mathcal{A}_\eta\restriction_\alpha,\mathcal{A}_\zeta\restriction_\alpha)$.
\end{proof}
Assume $T$ is a classifiable theory. We can conclude from the previous results that, $\eta\cong_T\xi$ if and only if $\eta\ R^\alpha_{\text{EF}}\ \xi$ for $\mu$-club many $\alpha$.
This lead us to the main result of this section, $\cong_T$ is continuously reducible to $E^\kappa_{\mu\text{-club}}$ for any $\mu$ when $T$ is classifiable.
\begin{teorema}
Assume $T$ is a classifiable theory and $\mu<\kappa$ a regular cardinal, then $\cong_T$ is continuously reducible to $E^\kappa_{\mu\text{-club}}$ ($\cong_T\ \leq_c\ E^\kappa_{\mu\text{-club}}$). 
\end{teorema}
\begin{proof}
Shelah proved \cite{Sh}, that if $T$ is classifiable then every two models of $T$ that are $L_{\infty , \kappa}$-equivalent are isomorphic. But $L_{\infty , \kappa}$-equivalent is equivalent to EF$^\kappa_\omega$-equivalence. In other words, if $T$ is classifiable then $\bf{II}\uparrow$ EF$^\kappa_\omega (\mathcal{A},\mathcal{B})\Longleftrightarrow \mathcal{A}\cong \mathcal{B}$. This game is a determined game, so $\bf{I}\uparrow$ EF$^\kappa_\omega (\mathcal{A},\mathcal{B})\Longleftrightarrow \mathcal{A}\ncong \mathcal{B}$.\\ \\
Define the reduction $\mathcal{F}:\kappa^\kappa\rightarrow\kappa^\kappa$ as follows,
$$
\mathcal{F}(\eta)(\alpha) = \begin{cases} f_\eta(\alpha) &\mbox{if } cf(\alpha)=\mu, \mathcal{A}_\eta\restriction_\alpha\models T \text{ and } R_{EF}^\alpha\text{ is an equivalence relation}\\
0 & \mbox{in other case } \end{cases} 
$$
where $f_\eta(\alpha)$ is a code in $\kappa\backslash \{0\}$ for the $R_{EF}^\alpha$ equivalence class of $\mathcal{A}_\eta\restriction_\alpha$\\ \\
First, we will show that $\mathcal{F}(\eta)\ E^\kappa_{\mu\text{-club}}\ \mathcal{F}(\xi)$ implies $\eta\cong_T\xi$. Assume $\eta$ and $\xi$ are such that $\mathcal{F}(\eta)\ E^\kappa_{\mu\text{-club}}\ \mathcal{F}(\xi)$.\\
It is known that if $\mathcal{A}$ is a model of $T$, then the set $\{\alpha<\kappa : \mathcal{A}\restriction_\alpha\models T\}$ contains a club. 
Therefore, if there are $\mu$-club many $\alpha$ such that $\mathcal{F}(\eta)(\alpha)=0$, then $\mathcal{A}_\eta\not\models T$, otherwise we will have a club disjoint to a $\mu$-club. So, if there are $\mu$-club many $\alpha$ satisfying $\mathcal{F}(\eta)(\alpha)=\mathcal{F}(\xi)(\alpha)=0$, then $\mathcal{A}_\eta\not\models T$ and $\mathcal{A}_\xi\not\models T$, giving us $\eta\cong_T \xi$.\\
On the other hand, if there are $\mu$-club many $\alpha$ satisfying $\mathcal{F}(\eta)(\alpha)=\mathcal{F}(\xi)(\alpha)\neq 0$, then there are $\mu$-club many $\alpha$ such that $\mathcal{A}_\eta\restriction_\alpha\models T$ and $\mathcal{A}_\xi\restriction_\alpha\models T$ and thus $\mathcal{A}_\eta\models T$ and $\mathcal{A}_\xi\models T$. Since there are $\mu$-club many $\alpha$ such that $\mathcal{A}_\eta\restriction_\alpha\models T$, $\mathcal{A}_\xi\restriction_\alpha\models T$ and $\bf{II}\uparrow$ EF$^\kappa_\omega(\mathcal{A}_\eta\restriction_\alpha,\mathcal{A}_\xi\restriction_\alpha)$, then by Corollary 2.5, $\bf{II}\uparrow$ EF$^\kappa_\omega (\mathcal{A}_\eta,\mathcal{A}_\xi)$ and $\eta\cong_T \xi$.\\ \\
To show that $\eta\cong_T\xi$ implies $\mathcal{F}(\eta)\ E^\kappa_{\mu\text{-club}}\ \mathcal{F}(\xi)$, assume  that $\eta$ and $\xi$ are such that $\eta\cong_T\xi$.\\
For the case when $\mathcal{A}_\eta\models T$, it is clear that $\mathcal{A}_\xi\models T$. We will show the existence of a $\mu$-club, such that for every element $\alpha$ of it, $f_\eta(\alpha)=f_\xi(\alpha)$. 
Notices that $\mathcal{A}_\eta\restriction_\alpha$ and $\mathcal{A}_\xi\restriction_\alpha$ are models of $T$ for club many $\alpha$. Since $\bf{II}\uparrow$ EF$^\kappa_\omega (\mathcal{A}_\eta,\mathcal{A}_\xi)\Longleftrightarrow \mathcal{A}_\eta\cong_T \mathcal{A}_\xi$, by Corollary 2.5 there are $\mu$-club many $\alpha$ such that $\bf{II}\uparrow$ EF$^\kappa_\omega(\mathcal{A}_\eta\restriction_\alpha,\mathcal{A}_\xi\restriction_\alpha)$, what is the same as $\eta\ R_{EF}^\alpha\ \xi$. Therefore, by Lemma 2.7 there is a $\mu$-club, such that for every $\alpha$ in it $f_\eta(\alpha)=f_\xi(\alpha)$.\\
For the case when $\mathcal{A}_\eta\not\models T$, since we assumed $\eta\cong_T \xi$, then $\mathcal{A}_\xi\not\models T$. There is $\varphi\in T$ such that $\mathcal{A}_\eta\models \neg\varphi$ and $\mathcal{A}_\xi\models \neg\varphi$. Therefore, there are club many $\alpha$ such that $\mathcal{A}_\eta\restriction_\alpha\models \neg\varphi$ and $\mathcal{A}_\xi\restriction_\alpha\models \neg\varphi$, in particular exist club many $\alpha$ such that $\mathcal{F}(\eta)(\alpha)=\mathcal{F}(\xi)(\alpha)=0$.\\ \\
To show that $\mathcal{F}$ is continuous, let $[\eta\restriction_\alpha]$ be a basic open set and $\xi\in \mathcal{F}^{-1}[[\eta\restriction_\alpha]]$. Let $\pi$ be the bijection in Definition 2.1, since $\kappa$ is regular, $sup\{\pi(a)|a\in \omega\times\alpha^{<\omega}\}<\kappa$. Therefore, there is $\beta>\alpha$ such that for every $\gamma<\alpha$ holds $$(a_1,a_2,\ldots , a_n)\in P_m^{\mathcal{A}_\xi\restriction_{\gamma}}\Leftrightarrow n=g_1^{-1}(m) \text{ and }\xi\restriction_\beta(\pi(m,a_1,a_2,\ldots,a_n))>0.$$
Then, for every $\zeta\in [\xi\restriction_\beta]$ and $\gamma<\alpha$, $\mathcal{A}_\xi\restriction_{\gamma}$ and $\mathcal{A}_\zeta\restriction_{\gamma}$ are isomorphic. So for every $\zeta\in [\xi\restriction_\beta]$, $f_\xi(\gamma)=f_\zeta(\gamma)$ for every $\gamma<\alpha$. We conclude that $[\xi\restriction_\beta]\subseteq \mathcal{F}^{-1}[[\eta\restriction_\alpha]]$ and $\mathcal{F}$ is continuous.
\end{proof}


\section{Stable Unsuperstable Theories}
A set $X\subset \kappa^\kappa$ is $\Sigma_1^1$ if it is the projection of a Borel set $C\subset \kappa^\kappa\times\kappa^\kappa$, notice that $\kappa^\kappa\times \kappa^\kappa$ is homeomorphic to $\kappa^\kappa$. Let $X\in\{\lambda^\kappa|1<\lambda\leq \kappa\}$ and we think this as subspaces of $\kappa^\kappa$. We say that an equivalence relation $E$ on $X$ is $\Sigma_1^1$-complete, if it is $\Sigma_1^1$ (as a subset of $\kappa^\kappa\times\kappa^\kappa$) and for every $\Sigma_1^1$-equivalence relation $F$ on a space $Y\in\{\lambda^\kappa|1<\lambda\leq \kappa\}$, there is a Borel reduction $F\leq_B E$.\\ \\
On the works \cite{FHK14}, \cite{FHK13} and \cite{HK}, the relation $E_{\mu\text{-club}}^\lambda$, $1<\lambda\leq \kappa$, has been studied on the closed subspaces $\lambda^\kappa$, with $\lambda<\kappa$ and the relative subspace topology. The relation $E_{\mu\text{-club}}^\lambda$ on the subspace $\lambda^\kappa$ is defined as: we say that $f,g\in \lambda^\kappa$ are $E^\lambda_{\mu\text{-club}}$ equivalent ($f\ E^\gamma_{\mu\text{-club}}\ g$) if the set $\{\alpha<\kappa|f(\alpha)=g(\alpha)\}$ contains a $\mu$-club. For these relations the following results are known:\\ \\
{\bf Theorem.} \textit{(\cite{FHK13}) If a first order theory $T$ is classifiable, then for all regular $\mu<\kappa$, $E^2_{\mu\text{-club}} \nleq_B \cong_T$.}\\ \\
{\bf Theorem.} \textit{(\cite{FHK13}) Suppose that $\kappa=\lambda^+=2^\lambda$ and $\lambda^{<\lambda}=\lambda$.}
\begin{itemize}
\item \textit{If $\kappa>2^\omega$ and $T$ is a first order theory, then $T$ is classifiable if and only if for all regular $\mu<\kappa$, $E^2_{\mu\text{-club}} \nleq_B \cong_T$.}
\item \textit{If $T$ is unstable then $E^2_{\lambda\text{-club}} \leq_c \cong_T$.}
\end{itemize}
{\bf Theorem.} \textit{(\cite{HK}) Assume $V=L$. Suppose $\kappa>\omega$.}
\begin{itemize}
\item \textit{If $\kappa=\lambda^+$, then for every regular cardinal $\mu$, the equivalence relation $E^\lambda_{\mu\text{-club}}$ is $\Sigma_1^1$-complete.}
\item \textit{If $\kappa$ is inaccessible, then for every regular cardinal $\mu$, the equivalence relation $E^\kappa_{\mu\text{-club}}$ is $\Sigma_1^1$-complete.}
\end{itemize}
{\bf Theorem.} \textit{(\cite{FHK14}) Suppose $T$ is a classifiable and shallow theory and $\kappa>2^\omega$, then for all regular $\mu<\kappa$, $\cong_T\ \leq_B\ E_{\mu\text{-club}}^\kappa$.}\\ \\
Some of the results are specifically for some fix theory. Let $\alpha$ be a countable ordinal, define $T_{\alpha}=Th((\omega^{\alpha},R_\beta)_{\beta<\alpha})$, where $\eta\ R_\beta\ \xi$ holds if $\eta\restriction_{\beta}=\xi\restriction_{\beta}$. \\ \\
{\bf Theorem.} \textit{(\cite{HK}) Assume $V=L$. If $\kappa=\lambda^+$ and $\lambda$ is regular, then $E_{\omega\text{-club}}^\lambda\leq_B \cong_{T_{\omega+\omega}}$.} \\ \\
We are going to continue with this work, reducing $E_{\omega\text{-club}}^\kappa$ to some other equivalence relations and generalize some of these results. We will use similar ideas as the ones used on \cite{FHK14}, \cite{FHK13} and \cite{HK}.
\begin{teorema} (\cite{FHK13})
Suppose for all $\gamma<\kappa$, $\gamma^\omega<\kappa$ and $T$ is a stable unsuperstable theory. Then $E^2_{\omega\text{-club}}\leq_c \cong_T$.
\end{teorema}
Given an equivalence relation $E$ on $X$ it is natural to think on a $\lambda$-product relation of it for any $0<\lambda<\kappa$. The $\lambda$-product relation $\Pi_{\lambda}E$, is the relation defined on $X^\lambda\times X^\lambda$ as, $f\ \Pi_{\lambda}E\ g$ if $f_\gamma\ E\ g_\gamma$ holds for every $\gamma<\lambda$, where $f=(f_\gamma)_{\gamma<\lambda}$ and $g=(g_\gamma)_{\gamma<\lambda}$. We will work on the space $(2^\kappa)^\lambda$, with the box topology on $(2^\kappa)^\lambda$, the topology generated by the basic open sets $\{\Pi_{\alpha<\lambda}\mathcal{O}_\alpha|\forall\alpha<\lambda(\mathcal{O}_\alpha\text{ is an open set in }2^\kappa)\}$.\\ \\
{\bf Remark.} If there exists a cardinal $\lambda<\kappa$ such that $\kappa=2^\lambda$, the relations $E^\kappa_{\mu\text{-cub}}$ and $\Pi_{\lambda}E^2_{\mu\text{-cub}}$ are bireducible. \\
Let $G$ be a bijection between $\kappa$ and $2^\lambda$. Define $\mathcal{F}:\kappa^\kappa\rightarrow(\kappa^\kappa)^\lambda$, by $\mathcal{F}(f)=(f_\gamma)_{\gamma<\lambda}$, where $f_\gamma(\alpha)=G(f(\alpha))(\gamma)$ for every $\gamma<\lambda$ and $\alpha<\kappa$. $\mathcal{F}$ is a reduction of $E^\kappa_{\mu\text{-cub}}$ to $\Pi_{\lambda}E^2_{\mu\text{-cub}}$. Clearly for every pair of function $f$ and $g$ in $\kappa^\kappa$, $f(\alpha)=g(\alpha)$ implies $G(f(\alpha))=G(g(\alpha))$ and $f_\gamma(\alpha)=g_\gamma(\alpha)$ for every $\gamma<\lambda$. Therefore, if $f$ and $g$ coincide in a $\mu$-club, then for all $\gamma<\lambda$, $f_\gamma$ and $g_\gamma$ coincide in the same $\mu$-club. For the other direction, assume that $f_\gamma$ and $g_\gamma$ coincide in a $\mu$-club for every $\gamma<\lambda$. Since the intersection of less than $\kappa$ $\mu$-club sets is a $\mu$-club set, then there is a $\mu$-club $C$, in which the functions $f_\gamma$ and $g_\gamma$ coincide for every $\gamma<\lambda$. Therefore $G(f(\alpha))(\gamma)=G(g(\alpha))(\gamma)$ for every $\gamma<\lambda$ and every $\alpha\in C$. So $G(f(\alpha))=G(g(\alpha))$ for every $\alpha\in C$ and since $G$ is a bijection, we can conclude that $f(\alpha)=g(\alpha)$ for every $\alpha\in C$.\\ 
The other reduction is proved in \cite{FHK14}.\\ \\
A nice example of a stable unsuperstable theory is $T_\omega$. Under the assumptions of Theorem 3.1, $E^2_{\omega\text{-cub}}\ \leq_c\ \cong_{T_\omega}$. This and the reducibility of $E^\kappa_{\mu\text{-cub}}$ to $\Pi_{\lambda}E^2_{\mu\text{-cub}}$ lead us to our first reduction related to stable unsuperstable theories.
\begin{lema}
Suppose that for all $\gamma<\kappa$, $\gamma^\omega<\kappa$ and $\kappa=2^\lambda$. Then $E^\kappa_{\omega\text{-cub}}\ \leq_c\ \cong_{T_\omega}$.
\end{lema}
\begin{proof}
By the previous remark it is enough to prove $\Pi_{\lambda}E^2_{\omega\text{-club}}\ \leq_c\  \cong_{T_\omega}$. Let $(\mathcal{A}_{\alpha})_{\alpha<\lambda}$ be pairwise non isomorphic models of $T_\omega$ with universe $\kappa$. Let $F$ be a continuous reduction of $E^2_{\omega\text{-cub}}$ to $\cong_{T_\omega}$.\\
For every $f=(f_\gamma)_{\gamma<\lambda}\in (2^\kappa)^\lambda$ we will define the model $\mathcal{A}^f$, with domain $\lambda\times (\omega+1)\times\kappa$. The interpretation of the relation $R_0^{\mathcal{A}^f}$ is the following, $(\gamma_1,\beta_1,\alpha_1)\ R_0^{\mathcal{A}^f}\ (\gamma_2,\beta_2,\alpha_2)$ if and only if $\gamma_1=\gamma_2$. The interpretation of the relations $R_{i}^{\mathcal{A}^f}$ (for $0<i$) is the following, $(\gamma_1,\beta_1,\alpha_1)\ R_i^{\mathcal{A}^f}\ (\gamma_2,\beta_2,\alpha_2)$ if and only if $\gamma_1=\gamma_2$, $\beta_1=\beta_2$ and $\alpha_1\ R_i^{\mathcal{A}}\ \alpha_2$ where $\mathcal{A}=\mathcal{A}_{\gamma_1}$ if $\beta< \omega$ otherwise $\mathcal{A}=\mathcal{A}_{F(f_{\gamma_1})}$.\\
\begin{claim}
$f\ \Pi_{\lambda}E^2_{\omega\text{-club}}\ g$ if and only if $\mathcal{A}^{f}$ and $\mathcal{A}^g$ are isomorphic.
\end{claim}
\textit{Proof of the claim.} 
Let $f=(f_\gamma)_{\gamma<\lambda}$ and $g=(g_\gamma)_{\gamma<\lambda}$.
If $f\ \Pi_\lambda E^2_{\omega\text{-club}}\ g$, then $f_\gamma\  E^2_{\omega\text{-club}}\ g_\gamma$ for every $\gamma<\lambda$, therefore for every $\gamma<\lambda$ the models $\mathcal{A}_{F(f_\gamma)}$ and $\mathcal{A}_{F(g_\gamma)}$ are isomorphic. Let $H_\gamma$ be an isomorphism between $\mathcal{A}_{F(f_\gamma)}$ and $\mathcal{A}_{F(g_\gamma)}$ for every $\gamma<\lambda$, define $H:\mathcal{A}^{f}\rightarrow \mathcal{A}^{g}$ by,
$$
H(\gamma,\alpha,\beta) = \begin{cases} (\gamma,\omega, H_\gamma (\beta)) &\mbox{if } \alpha=\omega\\
(\gamma,\alpha,\beta) & \mbox{in other case. } \end{cases} 
$$
It is clear that $H$ is an isomorphism between $\mathcal{A}^{f}$ and $\mathcal{A}^{g}$.\\ \\
Assume there exists an isomorphism $H:\mathcal{A}^{f}\rightarrow \mathcal{A}^{g}$. Fix $\gamma<\lambda$, since for every $\beta_1$ and $\beta_2$ in $\omega+1$, and $\alpha_1$ and $\alpha_2$ in $\kappa$, $(\gamma,\beta_1,\alpha_1)\ R_0^{\mathcal{A}^f}\ (\gamma,\beta_2,\alpha_2)$ if and only if \\$H(\gamma,\beta_1,\alpha_1)\ R_0^{\mathcal{A}^g}\ H(\gamma,\beta_2,\alpha_2)$, then $H(\{\gamma\}\times (\omega+1)\times \kappa)\subseteq \{\alpha\}\times (\omega+1) \times \kappa$ for some $\alpha<\lambda$. Following the same argument, we can conclude that $H^{-1}(\{\alpha\}\times (\omega+1) \times \kappa)\subseteq \{\gamma\}\times (\omega+1) \times \kappa$. Therefore $\mathcal{A}^f\restriction_{\{\gamma\}\times \{n\} \times \kappa}$ and $\mathcal{A}^g\restriction_{\{\alpha\}\times \{m\} \times \kappa}$ are isomorphic for some $n,m\in\omega$, so $\mathcal{A}_{\gamma}$ and $\mathcal{A}_{\alpha}$ are isomorphic. By the way $\mathcal{A}^f$ and $\mathcal{A}^g$ were constructed, this only happens when $\gamma=\alpha$. Then $H(\{\gamma\}\times (\omega+1) \times \kappa)= \{\gamma\}\times (\omega+1) \times \kappa$.
Since $H$ is an isomorphism, either $H(\{\gamma\}\times \omega \times \kappa)= \{\gamma\}\times \omega \times \kappa$ or there is a $n<\omega$ such that $H(\{\gamma\}\times \{n\} \times \kappa)=\{\gamma\}\times \{\omega\} \times \kappa$. For the first case, we can conclude that $H(\{\gamma\}\times \{\omega\} \times \kappa)=\{\gamma\}\times \{\omega\} \times \kappa$, then $\mathcal{A}_{F(f_\gamma)}$ and $\mathcal{A}_{F(g_\gamma)}$ are isomorphic. For the second case, $\mathcal{A}^f\restriction_{\{\gamma\}\times\{n\}\times \kappa}$ and $\mathcal{A}^g\restriction_{\{\gamma\}\times\{\omega\}\times \kappa}$ are isomorphic and there is $m<\omega$ such that $H(\{\gamma\}\times \{\omega\} \times \kappa)=\{\gamma\}\times \{m\} \times \kappa$. So $\mathcal{A}^f\restriction_{\{\gamma\}\times\{\omega\}\times \kappa}$ and $\mathcal{A}^g\restriction_{\{\gamma\}\times\{m\}\times \kappa}$ are isomorphic. By the way $\mathcal{A}^f$ and $\mathcal{A}^g$ were defined, we know that $\mathcal{A}^f\restriction_{\{\gamma\}\times \{n\} \times \kappa}$ and $\mathcal{A}^g\restriction_{\{\gamma\}\times \{m\} \times \kappa}$ are isomorphic, therefore $\mathcal{A}^f\restriction_{\{\gamma\}\times\{\omega\}\times \kappa}$ and $\mathcal{A}^g\restriction_{\{\gamma\}\times\{\omega\}\times \kappa}$ are isomorphic (i.e. $\mathcal{A}_{F(f_\gamma)}$ and $\mathcal{A}_{F(g_\gamma)}$ are isomorphic). From the way $F$ was chosen we can conclude that $f_\gamma\ E^2_{\omega\text{-cub}}\ g_\gamma$. And so for all $\gamma<\lambda$, $f_\gamma\ E^2_{\omega\text{-cub}}\ g_\gamma$ and finally we conclude that $f\ \Pi_{\lambda}E^2_{\omega\text{-club}}\ g$.  \hfill $_{\square_{\text{Claim 3.3}}}$\\ \\
Let $\sigma$ be a bijection from $\lambda\times (\omega+1)\times \kappa$ to $\kappa$, let $\pi$ and $P_n$ be as in Definition 2.1. We define the reduction $\mathcal{F}:(\kappa^\kappa)^\lambda\rightarrow \kappa^\kappa$ by, 
$$
\mathcal{F}((f_\gamma)_{\gamma<\lambda})(\alpha) = \begin{cases} 1 &\mbox{if } \alpha=\pi(n,a_1,a_2) \text{ and } \mathcal{A}^f\models P_n(\sigma^{-1}(a_1),\sigma^{-1}(a_2))\\
0 & \mbox{in other case. } \end{cases} 
$$
The continuity of $\mathcal{F}$, can be proved as in the proof of Theorem 2.8.
\end{proof}
The following corollary follows from Theorem 2.8 and Lemma 3.2. 
\begin{corolario}
Suppose for all $\gamma<\kappa$, $\gamma^\omega<\kappa$ and $\kappa=2^\lambda$, $\lambda<\kappa$. If $T$ is a classifiable theory. Then $\cong_{T}\ \leq_c\ \cong_{T_\omega}$.
\end{corolario}


\section{Coloured trees}
In this section we will define the coloured trees. These trees have high $\omega+2$ and a colouring function. We will show a construction of a coloured tree, using an element of $\kappa^\kappa$ to define the colouring function. In the end these trees are going to be isomorphic if and only if their respective elements of $\kappa^\kappa$ used to construct them are $E^\kappa_{\omega\text{-cub}}$ related. This is Lemma 4.7, below, but notice that in section 5 we need more information about the trees than just this lemma.\\
The coloured trees that we will present in this section, are a variation of the trees used in \cite{HK} and \cite{FHK13} for the reduction mentioned at the beginning of the previous section.\\
For every $x\in t$ we denote by $ht(x)$ the height of $x$, the order type of $\{y\in t|y<x\}$. Define $t_\alpha=\{x\in t|ht(x)=\alpha\}$ and denote by $x\restriction_\alpha$ the unique $y\in t$ such that $y\in t_\alpha$ and $y\leq x$. An $\alpha, \beta$-tree is a tree $t$ in which every element has less than $\alpha$ immediately successors and every branch $\eta$ has order type less than $\beta$.
\begin{definicion} 
A coloured tree is a pair $(t,c)$, with $t$ is a $\kappa^+$, $(\omega+2)$-tree and $c$ is a map $c:t_\omega\rightarrow \kappa\backslash \{0\}$. 
\end{definicion}
\noindent
Two coloured trees $(t,c)$ and $(t',c')$ are isomorphic, if there is a trees isomorphism $f:t\rightarrow t'$ such that for every $x\in t_\omega$, $c(x)=c'(f(x))$.\\
Denote the set of all coloured trees by $CT^\omega$. Let $CT^\omega_*\subset CT^\omega$ be the set of coloured trees, in which every element with finite height, has infinitely many immediate successors, every maximal branch has order type $\omega+1$ and the intersection of two distinct branches is finite. Notice that for every $t\in CT^\omega_*$ and every pair $x,y\in t_{\omega}$, $x\restriction_\omega=y\restriction_\omega$ implies $x=y$.\\
We are going to work only with elements of $CT^\omega_*$, every time we mention a coloured tree, we mean an element of $CT^\omega_*$.\\
We can see every coloured tree as a downward closed subset of $\kappa^{\leq \omega}$.
\begin{definicion}
Let $(t,c)$ be a coloured tree, suppose $(I_\alpha)_{\alpha<\kappa}$ is a collection of subsets of $t$ that satisfies:
\begin{itemize}
\item for each $\alpha<\kappa$, $I_\alpha$ is a downward closed subset of $t$.
\item $\bigcup_{\alpha<\kappa}I_\alpha=t$.
\item if $\alpha<\beta<\kappa$, then $I_\alpha\subset I_\beta$.
\item if $\gamma$ is a limit ordinal, then $I_\gamma=\bigcup_{\alpha<\gamma}I_\alpha$.
\item for each $\alpha<\kappa$ the cardinality of $I_\alpha$ is less than $\kappa$.
\end{itemize}
We call $(I_\alpha)_{\alpha<\kappa}$ a filtration of $t$.
\end{definicion}
\begin{definicion}
Let $t$ be a coloured tree and $\mathcal{I}=(I_\alpha)_{\alpha<\kappa}$ a filtration of $t$. Define $H_{\mathcal{I},t}\in \kappa^\kappa$ as follows. \\
Fix $\alpha<\kappa$. Let $B_\alpha$ be the set of all $x\in t_\omega$ that are not in $I_\alpha$, but $x\restriction_n\in I_\alpha$ for all $n<\omega$.
\begin{itemize}
\item If $B_\alpha$ is non-empty and there is $\beta$ such that for all $x\in B_\alpha$, $c(x)=\beta$, then let $H_{\mathcal{I},t}(\alpha)=\beta$
\item Otherwise let $H_{\mathcal{I},t}(\alpha)=0$
\end{itemize}
\end{definicion}
We will call a filtration good if for every $\alpha$, $B_\alpha\neq \emptyset$ implies that $c$ is constant on $B_\alpha$.
\begin{lema}
Suppose $(t_0,c_0)$ and $(t_1,c_1)$ are isomorphic coloured trees, and $\mathcal{I}=(I_\alpha)_{\alpha<\kappa}$ and $\mathcal{J}=(J_\alpha)_{\alpha<\kappa}$ are good filtrations of $(t_0,c_0)$ and $(t_1,c_1)$ respectively. Then $H_{\mathcal{I},t_0}\ E^\kappa_{\omega\text{-club}}\  H_{\mathcal{J},t_1}$
\end{lema}
\begin{proof}
Let $F:(t_0,c_0)\rightarrow (t_1,c_1)$ be a coloured tree isomorphism. Define $F\mathcal{I}=(F[I_\alpha])_{\alpha<\kappa}$. It is easy to see that $F[I_\alpha]$ is a downward closed subset of $t_1$. Clearly $F[I_\alpha]\subset F[I_\beta]$ when $\alpha<\beta$ and for $\gamma$ a limit ordinal, $\cup_{\alpha<\gamma}F[I_\alpha]=F[I_\gamma]$. If $x\in t_1$ then there exists $y\in t_0$ and $\alpha<\kappa$ such that $F(y)=x$ and $y\in I_\alpha$, therefore $x\in F[I_\alpha]$ and $\cup_{\alpha<\kappa}F[I_\alpha]=t_1$. Since $F$ is an isomorphism, $|F[I_\alpha]|=|I_\alpha|<\kappa$ for every $\alpha<\kappa$. So $F\mathcal{I}$ is a filtration of $t_1$.\\ For every $\alpha$, $B_\alpha^{\mathcal{I}}\neq \emptyset$ implies that $B_\alpha^{F\mathcal{I}}\neq \emptyset$. On the other hand, $\mathcal{I}$ is a good filtration, then when $B_\alpha^\mathcal{I}\neq \emptyset$, $c_0$ is constant on $B_\alpha^{\mathcal{I}}$. Since $F$ is colour preserving, $c_1$ is constant on $B_\alpha^{F\mathcal{I}}$, we conclude that $F\mathcal{I}$ is a good filtration and $H_{\mathcal{I},t_0}(\alpha)=H_{F\mathcal{I},t_1}(\alpha)$.\\
Notice that $F[I_\alpha]=J_\alpha$ implies $H_{\mathcal{I},t_0}(\alpha)=H_{\mathcal{J},t_1}(\alpha)$. Therefore it is enough to show that $C=\{\alpha|F[I_\alpha]=J_\alpha\}$ is an $\omega$-club. By the definition of a filtration, for every sequence $(\alpha_i)_{i<\theta}$ in $C$, cofinal to $\gamma$, $J_\gamma=\bigcup_{i<\theta}J_{\alpha_i}=\bigcup_{i<\theta}F[I_{\alpha_i}]=F[I_\gamma]$, so $C$ is closed. To show that $C$ is unbounded, choose $\alpha<\kappa$. Define the succession $(\alpha_i)_{i<\omega}$ by induction. For $i=0$, $\alpha_0=\alpha$. When $n$ is odd let $\alpha_{n+1}$ be the least ordinal bigger than $\alpha_n$ such that $F[I_{\alpha_n}]\subset J_{\alpha_{n+1}}$ (such ordinal exists because $\kappa$ is regular, and $\mathcal{J}$ and $F\mathcal{I}$ are filtrations, specially $|F[I_{\alpha_n}]|<\kappa$). When $n$ is even let $\alpha_{n+1}$ be the least ordinal bigger than $\alpha_n$ such that $J_{\alpha_n}\subset F[I_{\alpha_{n+1}}]$ (such ordinal exists because $\kappa$ is regular, and $\mathcal{J}$ and $F\mathcal{I}$ are filtrations, specially $|J_{\alpha_n}|<\kappa$). Clearly $\bigcup_{i<\omega}J_{\alpha_i}=\bigcup_{i<\omega}F[I_{\alpha_i}]$ and  $\cup_{i<\omega}\alpha_i\in C$.
\end{proof}
Now we can construct the coloured trees that we need for the reduction. This construction is in essential the same used in \cite{HK}. The only difference between them is that in \cite{HK} the construction was made for successor cardinals, here we do it for inaccessible cardinals. These trees are useful for the study of the relation $E^\kappa_{\omega\text{-cub}}$.\\ \\
Order the set $\omega\times \kappa\times \kappa\times \kappa\times \kappa$ lexicographically, $(\alpha_1,\alpha_2,\alpha_3,\alpha_4,\alpha_5)>(\beta_1,\beta_2,\beta_3,\beta_4,\beta_5)$ if for some $1\leq k \leq 5$, $\alpha_k>\beta_k$ and for every $i<k$, $\alpha_i=\beta_i$. Order the set $(\omega\times \kappa\times \kappa\times \kappa\times \kappa)^{\leq \omega}$ as a tree by inclusion.\\
Define the tree $(I_f,d_f)$ as, $I_f$ the set of all strictly increasing functions from some $n\leq \omega$ to $\kappa$ and for each $\eta$ with domain $\omega$, $d_f(\eta)=f(sup(rang(\eta)))$.\\
For every pair of ordinals $\alpha$ and $\beta$, $\alpha<\beta<\kappa$ and $i<\omega$ define $$R(\alpha,\beta,i)=\bigcup_{i< j\leq \omega}\{\eta:[i,j)\rightarrow[\alpha,\beta)|\eta \text{ strictly increasing}\}.$$
\begin{definicion}
Assume $\kappa$ is an inaccessible cardinal. If $\alpha<\beta<\kappa$ and $\alpha,\beta,\gamma\neq 0$, let $\{P^{\alpha,\beta}_\gamma|\gamma<\kappa\}$ be an enumeration of all downward closed subtrees of $R(\alpha,\beta,i)$ for all $i$, in such a way that each possible coloured tree appears cofinally often in the enumeration. And the tree $P^{0,0}_0$ is $(I_f,d_f)$.\\
\end{definicion}
This enumeration is possible because $\kappa$ is inaccessible; there are at most\\ $|\bigcup_{i<\omega}\mathcal{P}(R(\alpha,\beta,i))|\leq \omega\times\kappa=\kappa$ downward closed coloured subtrees, and at most $\kappa\times \kappa^{<\kappa}=\kappa$ coloured trees.\\
Denote by $Q(P^{\alpha,\beta}_\gamma)$ the unique natural number $i$ such that $P^{\alpha,\beta}_\gamma\subset R(\alpha,\beta,i)$.
\begin{definicion}
Assume $\kappa$ is an inaccessible cardinal. 
Define for each $f\in \kappa^\kappa$ the coloured tree $(J_f,c_f)$ by the following construction.\\
For every $f\in \kappa^\kappa$ define $J_f=(J_f,c_f)$ as the tree of all $\eta: s\rightarrow \omega\times \kappa^4$, where $s\leq \omega$, ordered by extension, and such that the following conditions hold for all $i,j<s$:\\
Denote by $\eta_i$, $1<i<5$, the functions from $s$ to $\kappa$ that satisfies, $\eta(n)=(\eta_1(n),\eta_2(n),\eta_3(n),\eta_4(n),\eta_5(n))$.
\begin{enumerate}
\item $\eta\restriction_n\in J_f$ for all $n<s$.
\item $\eta$ is strictly increasing with respect to the lexicographical order on $\omega\times \kappa^4$.
\item $\eta_1(i)\leq \eta_1(i+1)\leq \eta_1(i)+1$.
\item $\eta_1(i)=0$ implies $\eta_2(i)=\eta_3(i)=\eta_4(i)=0$.
\item $\eta_1(i)<\eta_1(i+1)$ implies $\eta_2(i+1)\ge \eta_3(i)+\eta_4(i)$.
\item $\eta_1(i)=\eta_1 (i+1)$ implies $\eta_k (i)=\eta_k (i+1)$ for $k\in \{2,3,4\}$.
\item If for some $k<\omega$, $[i,j)=\eta_1^{-1}\{k\}$, then $$\eta_5\restriction_{[i,j)}\in P^{\eta_2(i),\eta_3(i)}_{\eta_4(i)}.$$
\noindent Note that 7 implies $Q(P^{\eta_2(i),\eta_3(i)}_{\eta_4(i)})=i$.
\item If $s=\omega$, then either 
\begin{itemize}
\item [(a)] there exists a natural number $m$ such that $\eta_1(m-1)<\eta_1(m)$, for every $k \ge m$ $\eta_1(k)=\eta_1(k+1)$, and the color of $\eta$ is determined by $P^{\eta_2(m),\eta_3(m)}_{\eta_4(m)}$: $$c_f(\eta)=c(\eta_5\restriction_{[m,\omega)})$$ where $c$ is the colouring function of $P^{\eta_2(m),\eta_3(m)}_{\eta_4(m)}$.\\
\end{itemize}
or
\begin{itemize}
\item [(b)] there is no such $m$ and then $c_f(\eta)=f(sup(rang(\eta_5)))$.
\end{itemize}
\end{enumerate}
\end{definicion}
\begin{lema}
Assume $\kappa$ is an inaccessible cardinal, then for every $f,g\in \kappa^\kappa$ the following holds $$f\  E^\kappa_{\omega\text{-club}}\ g \Leftrightarrow  J_f\cong J_g$$
\end{lema}
\begin{proof}
By Lemma 4.4, it is enough to prove the following properties of $J_f$
\begin{enumerate}
\item There is a good filtration $\mathcal{I}$ of $J_f$, such that $H_{\mathcal{I},J_f}\ E^\kappa_{\omega\text{-club}}\ f$.
\item If $f\  E^\kappa_{\omega\text{-club}}\ g$, then $J_f\cong J_g$.
\end{enumerate}
Notice that for any $k\in rang(\eta_1)$ if $\eta_5\restriction_{[i,j)}\in P^{\eta_2(i),\eta_3(i)}_{\eta_4(i)}$, when  $[i,j)=\eta_1^{-1}\{k\}$ and if $i+1<j$, then $\eta_5\restriction_{[i,j)}$ is strictly increasing. If $\eta_1(i)<\eta_1(i+1)$, by Definition 4.6 item 5, $\eta_2(i+1)\ge \eta_3(i)+\eta_4(i)$, so $\eta_5(i)<\eta_3(i)\leq \eta_2(i+1)\leq \eta_5(i+1)$. Thus $\eta_5$ is strictly increasing. If $\eta\restriction_n\in J_f$ for every $n$, then $\eta\in J_f$. Clearly every maximal branch has order type $\omega+1$, every chain $\eta\restriction_1\subset\eta\restriction_2\subset\eta\restriction_3\subseteq \cdots$ has a unique limit in the tree, and every element in a finite level has an infinite number of successors (at most $\kappa$), therefore $J_f\in CT^\omega_*$.\\
For each $\alpha<\kappa$ define $J_f^\alpha$ as $$J_f^\alpha=\{\eta\in J_f| rang(\eta)\subset \omega\times(\beta)^4\text{ for some }\beta<\alpha\}.$$
Suppose $rang(\eta_1)=\omega$. As it was mentioned before, $\eta_5$ is increasing and $sup(rang(\eta_3))\ge sup(rang(\eta_5))\ge sup(rang(\eta_2))$. By Definition 4.6 $sup(rang(\eta_2))\ge sup(rang(\eta_3))$ and $sup(rang(\eta_2))\ge sup(rang(\eta_4))$, this lead us to 
\begin{equation}
sup(rang(\eta_4))\leq sup(rang(\eta_3))=sup(rang(\eta_5))=sup(rang(\eta_2)).
\end{equation}
When $\eta\restriction_k\in J_f^\alpha$ holds for every $k\in \omega$, can be concluded that $sup(rang(\eta_5))\leq \alpha$, if in addition $\eta\notin J_f^\alpha$, then 
\begin{equation}
sup(rang(\eta_5))= \alpha.
\end{equation}
\begin{claim}
Suppose $\xi\in J_f^\alpha$ and $\eta\in J_f$. If $dom(\xi)<\omega$, $\xi\subsetneq \eta$ and for every $k$ in $dom(\eta)\backslash dom(\xi)$, $\eta_1(k)=\xi_1(max(dom(\xi)))$ and $\eta_1(k)>0$. Then $\eta\in J_f^\alpha$.
\end{claim}
\textit{Proof of the claim.} 
Assume $\xi,\eta\in J_f$ are as in the assumption. Let $\beta_i=\xi_i(max(dom(\xi)))$, for $i\in\{2,3,4\}$. Since $\xi\in J_f^\alpha$, then there exists $\beta<\alpha$ such that $\beta_2,\beta_3,\beta_4<\beta$. By Definition 4.6 item 6 for every $k\in dom(\eta)\backslash dom(\xi)$, $\eta_i(k)=\beta_i$ for $i\in \{2,3,4\}$. Therefore, by Definition 4.6 item 7 and the definition of $P^{\beta_2,\beta_3}_{\beta_4}$, we conclude $\eta_5(k)<\beta_3<\beta$, so $\eta\in J_f^\alpha$. \hfill $_{\square_{\text{Claim 4.8}}}$
\begin{claim}
$|J_f|=\kappa$, $\mathcal{J}=(J_f^\alpha)_{\alpha<\kappa}$ is a good filtration of $J_f$ and $H_{\mathcal{J},J_f}\ E^\kappa_{\omega\text{-club}}\ f$
\end{claim}
\textit{Proof of the claim.} Clearly $J_f=\cup_{\alpha<\kappa}J_f^\alpha$, $J_f^\alpha$ is a downward closed subset of $J_f$, and $J_f^\alpha\subset J_f^\beta$ when $\alpha<\beta$. Since $\kappa$ is inaccessible, we conclude $|J_f^\alpha|<\kappa$ and $|J_f|=\kappa$. Finally, when $\gamma$ is a limit ordinal
$$
\begin{array} {lcl} 
J_f^\gamma & = &\{\eta\in J_f|\exists \beta<\gamma(rang(\eta)\subset \omega\times (\beta)^4)\}\\
& = & \{\eta\in J_f|\exists\alpha<\gamma,\exists \beta<\alpha(rang(\eta)\subset \omega\times (\beta)^4)\}\\
& = & \bigcup_{\alpha<\gamma}J_f^\alpha
\end{array}
$$
Suppose $\alpha$ has cofinality $\omega$, and $\eta\in J_f\backslash J_f^\alpha$ satisfies $\eta\restriction_k\in J_f^\alpha$ for every $k<\omega$. By the previous claim, $\eta$ satisfies Definition 4.6 item 8 (a) only if $\eta_1(n)=0$ for every $n\in \omega$. So $\eta_1$, $\eta_2$, $\eta_3$ and $\eta_4$ are constant zero, and $c_f(\eta)=d_f(\eta_5)$, where $d_f$ is the colouring function of $P^{0,0}_0=I_f$, $c_f(\eta)=f(sup(rang(\eta_5)))$. When $\eta$ satisfies Definition 4.6 item 8 (b), $c_f(\eta)=f(sup(rang(\eta_5)))$.\\
In both cases, $c_f(\eta)=f(\alpha)$. Therefore, if $B_\alpha\neq \emptyset$ then $c_f$ is constant on $B_\alpha$ and $\mathcal{J}$ is a good filtration.\\
By Definition 4.3 and since $\mathcal{J}$ is a good filtration, $H_{\mathcal{J},J_f}(\alpha)=f(\alpha)$. \hfill $_{\square_{\text{Claim 4.9}}}$
\begin{claim}
If $f\ E^\kappa_{\omega\text{-club}}\ g$, then $J_f\cong J_g$.
\end{claim}
\textit{Proof of the claim.} Let $C'\subseteq \{\alpha<\kappa|f(\alpha)=g(\alpha)\}$ a $\omega$-club testifying $f\  E^\kappa_{\omega\text{-club}}\ g$, and let $C\supset C'$ be the closure of $C'$ under limits. By induction we are going to construct an isomorphism between $J_f$ and $J_g$.\\
We define continuous increasing sequences $(\alpha_i)_{i<\kappa}$ of ordinals and $(F_{\alpha_i})_{i<\kappa}$ of partial isomorphism from $J_f$ to $J_g$ such that:
\begin{itemize}
\item[a)] If $i$ is a successor, then $\alpha_i$ is a successor ordinal and there exists $\beta\in C$ such that $\alpha_{i-1}<\beta<\alpha_i$ and thus if $i$ is a limit, $\alpha_i\in C$.
\item[b)] Suppose that $i=\gamma+n$, where $\gamma$ is a limit ordinal or $0$, and $n<\omega$ is even. Then $dom(F_{\alpha_i})=J_f^{\alpha_i}$.
\item[c)] Suppose that $i=\gamma+n$, where $\gamma$ is a limit ordinal or $0$, and $n<\omega$ is odd. Then $rang(F_{\alpha_i})=J_g^{\alpha_i}$.
\item[d)] If $dom(\xi)<\omega$, $\xi\in dom (F_{\alpha_i})$, $\eta\restriction_{dom(\xi)}=\xi$ and for every $k\ge dom(\xi)$ $$\eta_1(k)=\xi_1(max (dom(\xi)))\text{ and } \eta_1(k)>0$$ then $\eta\in dom (F_{\alpha_i})$. Similar for $rang(F_{\alpha_i})$.
\item[e)] If $\xi\in dom(F_{\alpha_i})$ and $k<dom(\xi)$, then $\xi\restriction_{k}\in dom(F_{\alpha_i})$.
\item[f)] For all $\eta\in dom (F_{\alpha_i})$, $dom(\eta)=dom(F_{\alpha_i}(\eta))$.
\end{itemize}
For every ordinal $\alpha$ denote by $M(\alpha)$ the ordinal that is order isomorphic to the lexicographic order of $\omega\times \alpha^4$.\\ \\
\textit{\bf First step (i=0).}\\
Let $\alpha_0=\beta+1$ for some $\beta\in C$. Let $\gamma$ be an ordinal such that there is a coloured tree isomorphism $h: P^{0,M(\beta)}_\gamma\rightarrow J_f^{\alpha_0}$ and $Q(P^{0,M(\beta)}_\gamma)=0$. It is easy to see that such $\gamma$ exists, by the way our enumeration was chosen.\\
Since $P^{0,M(\beta)}_\gamma$ and $J_f^{\alpha_0}$ are closed under initial segments, then $|dom(h^{-1}(\eta))|=|dom(\eta)|$. Also both domains are intervals containing zero, therefore $dom(h^{-1}(\eta))=dom(\eta)$.\\
Define $F_{\alpha_0}(\eta)$ for $\eta\in J^{\alpha_0}_f$ as follows, let $F_{\alpha_0}(\eta)$ be the function $\xi$ with $dom(\xi)=dom(\eta)$, and for all $\kappa<dom(\xi)$
\begin{itemize}
\item $\xi_1(k)=1$
\item $\xi_2(k)=0$
\item $\xi_3(k)=M(\beta)$
\item $\xi_4(k)=\gamma$
\item $\xi_5(k)=h^{-1}(\eta)(k)$
\end{itemize}
To check that $\xi\in J_g$, we will check every item of Definition 4.6. Since $rang(F_{\alpha_0})=\{1\}\times\{0\}\times\{M(\beta)\}
\times\{\gamma\}\times P^{0,M(\beta)}_\gamma$, $\xi$ satisfies 1. Also $\xi_5=h^{-1}(\eta)\in P^{0,M(\beta)}_\gamma$, by definition of $P^{\alpha,\beta}_\gamma$, we now that $\xi_5$ is strictly increasing with respect to the lexicographic order, then $\xi$ satisfies item 2. Notice that $\xi$ is constant in every component except for $\xi_5$, therefore $\xi$ satisfies the items 3, 5, 6, 8 (a). Clearly $\xi_1(i)\neq 0$, so $\xi$ satisfies item 4. Notice that $[0,\omega)=\xi_1^{-1}(1)$ but $P^{\xi_2(k),\xi_3(k)}_{\xi_4(k)}=P^{0,M(\beta)}_\gamma$ for every $k$, therefore $\xi_5\in P^{\xi_2(0),\xi_3(0)}_{\xi_4(0)}$ and $\xi$ satisfies 7.\\
Let us show that the conditions a)-f) are satisfied, the conditions a) and c) are clearly satisfied. By the way $F_{\alpha_0}$ was defined, $dom(F_{\alpha_0})=J^{\alpha_0}_f$ and  $dom(\eta)=dom(F_{\alpha_0}(\eta))$, these are the conditions b), e) and f). Since $dom(F_{\alpha_0})=J^{\alpha_0}_f$, the Claim 4.8 implies d) for $dom(F_{\alpha_0})$. For d) with $rang(F_{\alpha_0})$, suppose $\xi\in rang(F_{\alpha_0})$ and $\eta \in J_g$ are as in the assumption. Then $\eta_1(k)=\xi_1(k)=1$ for every $k<dom(\eta)$, by 6 in $J_g$ we have that $\eta_2(k)=\xi_2(k)=0$, $\eta_3(k)=\xi_3(k)=M(\beta)$ and $\eta_4(k)=\xi_4(k)=\gamma$ for every $k<dom(\eta)$. By 7 in $J_g$, $\eta_5\in P^{0,M(\beta)}_\gamma$ and since $rang(F_{\alpha_0})=\{1\}\times\{0\}\times\{M(\beta)\}
\times\{\gamma\}\times P^{0,M(\beta)}_\gamma$, we can conclude that $\eta\in rang(F_{\alpha_0})$.\\ \\
\textit{\bf Odd successor step.}\\
Suppose that $j<k$ is a successor ordinal such that $j=\beta_j+n_j$ for some limit ordinal (or 0) $\beta_j$ and an odd integer $n_j$. Assume  $\alpha_l$ and $F_{\alpha_l}$ are defined for every $l<j$ satisfying the conditions a)-f).\\
Let $\alpha_j=\beta+1$ where $\beta\in C$ is such that $\beta>\alpha_{j-1}$ and $rang(F_{\alpha_{j-1}})\subset J_g^\beta$, such a $\beta$ exists because $|rang(F_{\alpha_{j-1}})|\leq 2^{|\alpha_{j-1}|}$ and $\kappa$ is strongly inaccessible.\\ When $\eta\in rang(F_{\alpha_{j-1}})$ has finite domain $m$, define $$W(\eta)=\{\zeta|dom(\zeta)=[m,s),m<s\leq \omega, \eta^\frown \langle m,\zeta(m) \rangle\notin rang(F_{\alpha_{j-1}})\text{ and }\eta^\frown \zeta\in J_g^{\alpha_j}\}$$ with the color function $c_{W(\eta)}(\zeta)=c_g(\eta^\frown \zeta)$ for every $\zeta\in W(\eta)$ with $s=\omega$. Denote $\xi'=F^{-1}_{\alpha_{j-1}}(\eta)$, $\alpha=\xi'_3(m-1)+\xi'_4(m-1)$ and $\theta=\alpha+M(\alpha_j)$. Now choose an ordinal $\gamma_\eta$ such that $Q(P_{\gamma_\eta}^{\alpha,\theta})=m$ and there is an isomorphism $h_\eta:P_{\gamma_\eta}^{\alpha,\theta}\rightarrow W(\eta)$. We will define $F_{\alpha_{j}}$ by defining its inverse such that $rang(F_{\alpha_{j}})=J_g^{\alpha_j}$.\\
Each $\eta\in J_g^{\alpha_j}$ satisfies one of the followings:
\begin{itemize}
\item[(*)] $\eta\in rang(F_{\alpha_{j-1}})$.
\item[(**)] $\exists m<dom(\eta)(\eta\restriction_{m}\in rang(F_{\alpha_{j-1}})\wedge \eta\restriction_{(m+1)}\notin rang(F_{\alpha_{j-1}}))$.
\item[(***)] $\forall m<dom(\eta)(\eta\restriction_{(m+1)}\in rang(F_{\alpha_{j-1}})\wedge \eta\notin rang(F_{\alpha_{j-1}}))$.
\end{itemize}
We define $\xi=F_{\alpha_{j}}^{-1}(\eta)$ as follows. There are the three cases:\\ \\
Case $\eta$ satisfies $(*)$.\\
Define $\xi(n)=F_{\alpha_{j-1}}^{-1}(\eta)(n)$ for all $n<dom(\eta)$.\\ \\
Case $\eta$ satisfies $(**)$.\\
Let $m$ witnesses (**) for $\eta$. For every $n<dom(\xi)$
\begin{itemize}
\item If $n<m$, then $\xi(n)=F_{\alpha_{j-1}}^{-1}(\eta\restriction_m)(n)$.
\item For every $n\ge m$. Let
\begin{itemize}
\item $\xi_1(n)=\xi_1(m-1)+1$
\item $\xi_2(n)=\xi_3(m-1)+\xi_4(m-1)$
\item $\xi_3(n)=\xi_2(m)+M(\alpha_j)$
\item $\xi_4(n)=\gamma_{\eta\restriction_m}$
\item $\xi_5(n)=h^{-1}_{\eta\restriction_m}(\eta\restriction_{[m,dom(\eta))})(n)$
\end{itemize}
\end{itemize}
Note that, $\eta\restriction_{[m,dom(\eta))}$ is an element of $W({\eta\restriction_m})$, this makes possible the definition of $\xi_5$.\\
Let us check the items of Definition 4.6 to see that $\xi\in J_f$. Clearly item 1 is satisfied. By induction hypothesis, $\xi\restriction_m$ is increasing, $\xi_1(m)=\xi_1(m-1)+1$ so $\xi(m-1)<\xi(m)$, and $\xi_k$ is constant on $[m, \omega)$ for $k\in \{1,2,3,4\}$,since $h^{-1}_{\eta\restriction_m}(\eta)\in P_{\gamma_\eta}^{\alpha,\theta}$, then $\xi_5$ is increasing, and we conclude that $\xi$ is increasing respect to the lexicographic order, so $\xi$ satisfies item 2. Also we conclude $\xi_1(i)\leq \xi_1(i+1)\leq \xi_1(i)+1$, so $\xi$ satisfies item 3. For every $i<\omega$, $\xi_1(i)=0$ implies $i<m$, so $\xi(i)=F_{\alpha_{j-1}}^{-1}(\eta\restriction_m)(i)$ and by the induction hypothesis $\xi$ satisfies item 4. By the induction hypothesis, for every $i+1<m$, $\xi_1(i)<\xi_1(i+1)$ implies $\xi_2(i+1)\ge \xi_3(i)+\xi_4(i)$, on the other hand $\xi_1(i)=\xi_1 (i+1)$ implies $\xi_k (i)=\xi_k (i+1)$ for $k\in \{2,3,4\}$, clearly $\xi_2(m)\ge \xi_3(m-1)+\xi_4(m-1)$ and $\xi_k (i)=\xi_k (i+1)$ for $i>m$ and $k\in\{2,3,4\}$, then $\xi$ satisfies items 5 and 6.\\
Suppose $[i,j)=\xi_1^{-1}(k)$  for some $k$ in $rang(\xi)$. Either $j<m$ or $m=i$. If $j<m$, by the induction hypothesis $\xi_5\restriction_{[i,j)}\in P_{\xi_4(i)}^{\xi_2(i),\xi_3(i)}$, if $[i,j)=[m,dom(\xi))$, then $\xi_5\restriction_{[i,j)}=h^{-1}_{\eta\restriction_m}(\eta\restriction_{[m,dom(\xi))})\in P_{\xi_4(m)}^{\xi_2(m),\xi_3(m)}$, $\xi$ thus satisfies item 7. Since $\xi$ is constant on $[m,\omega)$, $\xi$ satisfies 8 (a). Finally by item 8 (a) when $dom(\zeta)=\omega$, $c_f(\xi)=c(\xi_5\restriction_{[m,\omega)})$, where $c$ is the color of $P_{\xi_4(m)}^{\xi_2(m),\xi_3(m)}$. Since $\xi_5\restriction_{[m,\omega)}=h^{-1}_{\eta\restriction_m}(\eta\restriction_{[m,\omega)})$, $c_f(\xi)=c(h^{-1}_{\eta\restriction_m}(\eta\restriction_{[m,\omega)}))$ and since $h$ is an isomorphism, $c_f(\xi)=c_{W(\eta\restriction _m)}(\eta\restriction_{[m,\omega)})=c_g(\eta)$.\\ \\
Case $\eta$ satisfies $(***)$.\\
Clearly $dom(\eta)=\omega$, by the induction hypothesis and condition d), $rang(\eta)=\omega$, otherwise $\eta\in rang(F_{\alpha_{j-1}})$. Let $F^{-1}_{\alpha_j}(\eta)=\xi=\cup_{n<\omega}F^{-1}_{\alpha_{j-1}}(\eta\restriction_n)$, by the induction hypothesis, $\xi$ is well defined. Since for every $n<\omega$, $\xi\restriction_n\in J_f$, then $\xi\in J_f$. Let us check that $c_f(\xi)=c_g(\eta)$. First note that $\xi\notin J_f^{\alpha_{j-1}}$, otherwise by the induction hypothesis f), 
\begin{equation*}
F_{\alpha_{j-1}}(\xi)=\bigcup_{n<\omega}F_{\alpha_{j-1}}(\xi\restriction_n)=\bigcup_{n<\omega}\eta\restriction_n=\eta 
\end{equation*}
giving us $\eta\in rang(F_{\alpha_{j-1}})$. By the equation (2), $sup(rang(\xi_5))=\alpha_{j-1}$ and $\xi$ satisfies item 8 b) in $J_f$, therefore $c_f(\xi)=f(\alpha_{j-1})$. Also by the definition of $J_f^\alpha$ and since $\xi\restriction_n\in J_f^{\alpha_{j-1}}$ for every $n<\omega$, $\alpha_{j-1}$ is a limit ordinal and by condition a), $j-1$ is a limit ordinal and $\alpha_{j-1}\in C$. The conditions b) and c) ensure $rang(F_{\alpha_{j-1}})=J_f^{\alpha_{j-1}}$. This implies, $\eta\notin J_f^{\alpha_{j-1}}$. By the equation (2), $sup(rang(\eta_5))=\alpha_{j-1}$. Therefore $\alpha_{j-1}$ has cofinality $\omega$, $\alpha_{j-1}\in C'$ and $f(\alpha_{j-1})=g(\alpha_{j-1})$. By item 8 b) in $J_g$, $c_g(\eta)=g(\alpha_{j-1})=f(\alpha_{j-1})=c_f(\xi)$.\\ \\
Next we show that $F_{\alpha_i}$ is a color preserving partial isomorphism. We already showed that $F_{\alpha_i}$ preserve the colors, so we only need to show that 
\begin{equation}
\eta\subsetneq \xi \Leftrightarrow F_{\alpha_i}^{-1}(\eta)\subsetneq F^{-1}_{\alpha_i}(\xi).
\end{equation}
From left to right.\\
When $\eta,\xi\in rang(F_{\alpha_{i-1}})$, the induction hypothesis implies (3) from left to right. If $\eta\in rang(F_{\alpha_{i-1}})$ and $\xi\notin rang(F_{\alpha_{i-1}})$, the construction implies (3) from left to right. Let us assume $\eta,\xi\notin rang(F_{\alpha_{i-1}})$, then $\eta,\xi$ satisfy (**). Let $m_1$ and $m_2$ be the respective natural numbers that witness (**) for $\eta$ and $\xi$, respectively. Notice that $m_2<dom(\eta)$, otherwise, $\eta\in rang(F_{\alpha_{i-1}})$. If $m_1<m_2$, clearly $\eta\in rang(F_{\alpha_{i-1}})$ what is not the case. A similar argument shows that $m_2<m_1$ cannot hold. We conclude that $m_1=m_2$ and by the construction of $F_{\alpha_i}$, $F_{\alpha_i}^{-1}(\eta)\subsetneq F^{-1}_{\alpha_i}(\xi)$.\\ \\
From right to left.\\
When $\eta,\xi\in rang(F_{\alpha_{i-1}})$, the induction hypothesis implies (3) from right to left. If $\eta\in rang(F_{\alpha_{i-1}})$ and $\xi\notin rang(F_{\alpha_{i-1}})$, the construction implies (3) from right left. Let us assume $\eta,\xi\notin rang(F_{\alpha_{i-1}})$, then $\eta,\xi$ satisfy (**). Let $m_1$ and $m_2$ be the respective natural numbers that witness (**) for $\eta$ and $\xi$, respectively. Notice that $m_2<dom(\eta)$, otherwise, $F_{\alpha_i}^{-1}(\eta)=F_{\alpha_{i-1}}^{-1}(\eta)$ and $\eta\in rang(F_{\alpha_{i-1}})$. If $m_1<m_2$, then 
$$
\begin{array} {lcl} 
F_{\alpha_i}^{-1}(\eta)_1(m_2-1) & = & (F_{\alpha_{i}}^{-1}(\xi)\restriction _{m_2})_1(m_2-1)\\ & < & F_{\alpha_{i}}^{-1}(\xi\restriction _{m_2})_1(m_2-1)+1\\ & = & F_{\alpha_i}^{-1}(\eta)_1(m_2)\\ & = & F_{\alpha_i}^{-1}(\eta)_1(m_2-1).
\end{array}
$$
This cannot hold. A similar argument shows that $m_2<m_1$ cannot hold. We conclude that $m_1=m_2$.\\
By the induction hypothesis $F_{\alpha_{i-1}}^{-1}(\eta\restriction_{m_1})=F_{\alpha_{i-1}}^{-1}(\xi\restriction_{m_2})$ implies $\eta\restriction_{m_1}=\xi\restriction_{m_2}$ (also implies $h_{\eta\restriction_{m_1}}=h_{\xi\restriction_{m_2}}$). Since $F_{\alpha_{i-1}}^{-1}(\eta\restriction_{m_1})(n)=F_{\alpha_i}^{-1}(\eta)(n)$ for all $n<m_1$, we only need to prove that $\eta\restriction_{[m_1,dom(\eta))}\subsetneq \xi\restriction_{[m_2,dom(\xi))}$. But $h_{\eta\restriction_{m_1}}$ is an isomorphism and $F_{\alpha_i}^{-1}(\eta)_5(n) = F_{\alpha_{i}}^{-1}(\xi)_5(n)$ for every $n\ge m_1$, so $h^{-1}_{\eta\restriction_{m_1}}(\eta\restriction_{[m_1,dom(\eta))})(n)=h^{-1}_{\xi\restriction_{m_2}}(\xi\restriction_{[m_2,dom(\xi))})(n)$. Therefore $\eta\restriction_{[m_1,dom(\eta))}\subsetneq \xi\restriction_{[m_2,dom(\xi))}$.\\ \\
Let us check that this three constructions satisfy the conditions a)-f).\\
When $i$ is a successor we have  $\alpha_{i-1}<\beta<\alpha_i=\beta+1$ for some $\beta\in C$, this is the condition a). Clearly the three cases satisfy b). We defined $F_{\alpha_i}^{-1}$ according to (*), (**), or (***); since every $\eta\in J_g^{\alpha_j}$ satisfies one of these, we conclude $rang(F_{\alpha_i})=J_g^{\alpha_j}$ which is the condition c).\\
Let us show that the $F_{\alpha_i}$ satisfy condition d). Let $\xi$ and $\beta$ be as in the assumptions of condition d) for domain. Notice that if $\xi\in dom(F_{\alpha_{i-1}})$ then the induction hypothesis and Claim 4.8, ensure that $\eta\in dom(F_{\alpha_i})$. Suppose $\xi\notin dom(F_{\alpha_{i-1}})$, then $F_{\alpha_{i}}(\xi)\notin rang(F_{\alpha_{i-1}})$. Since $dom(\xi)<\omega$, so $F_{\alpha_{i}}(\xi)$ satisfies (**). Let $m$ be the number witnessing it. Clearly $\xi\in J_f^{\alpha_i}$, by Claim 4.8 $\eta\in J_f^{\alpha_i}$. By item 6 in $J_f^{\alpha_i}$, $\eta_k$ is constant on $[m,dom(\eta))$ for $k\in \{2,3,4\}$, now by Definition 4.6 numeral 7 in $J_f^{\alpha_i}$, $\eta_5\restriction_{[m,dom(\eta))}\in P^{\alpha,\beta}_{\gamma_{\xi\restriction _m}}$. Let $\zeta=h_{\xi\restriction_m}(\eta_{[m,dom(\eta))})$, then $\eta=F^{-1}_{\alpha_{i}}(F_{\alpha_i}(\xi\restriction_m)^\frown \zeta)$ and $\eta\in dom(F_{\alpha_i})$.\\
The condition d) for range follows from Claim 4.8.\\
For the conditions e) and f), notice that $\xi$ was constructed such that $dom(\xi)=dom(\eta)$ and $\xi\restriction_k\in dom(F_{\alpha_i})$ which are these conditions.\\ \\
\textit{\bf Even successor step.}\\
Suppose that $j<k$ is a successor ordinal such that $j=\beta_j+n_j$ for some limit ordinal (or 0) $\beta_j$ and an even integer $n_j$. Assume $\alpha_l$ and $F_{\alpha_l}$ are defined for every $l<j$ satisfying conditions a)-f).\\
Let $\alpha_j=\beta+1$ where $\beta\in C$ such that $\beta>\alpha_{j-1}$ and $dom(F_{\alpha_{j-1}})\subset J_f^\beta$, such a $\beta$ exists because $|dom(F_{\alpha_{j-1}})|\leq 2^{|\alpha_{j-1}|}$ and $\kappa$ is strongly inaccessible. The construction of $F_{\alpha_j}$ such that $dom(F_{\alpha_j})=J_f^{\alpha_i}$ follows as in the odd successor step, with the equivalent definitions for $dom(F_{\alpha_j})$ and $J_f^{\alpha_i}$. Notice that for every $\eta\in J_f^{\alpha_j}$, there are only the following cases:
\begin{itemize}
\item[(*)] $\eta\in dom(F_{\alpha_{j-1}})$.
\item[(**)] $\exists m<dom(\eta)(\eta\restriction_{m}\in dom(F_{\alpha_{j-1}})\wedge \eta\restriction_{(m+1)}\notin dom(F_{\alpha_{j-1}}))$.
\end{itemize}
\textit{\bf Limit step.}\\
Assume $j$ is a limit ordinal. Let $\alpha_j=\cup_{i<j}\alpha_i$ and $F_{\alpha_j}=\cup_{i<j}F_{\alpha_i}$, clearly $F_{\alpha_j}:J_f^{\alpha_j}\rightarrow J_g$ and satisfies condition c). Since for $i$ successor, $\alpha_i$ is the successor of an ordinal in $C$, then $\alpha_j\in C$ and satisfies the condition a). Also $F_{\alpha_j}$ is a partial isomorphism. Remember that $\cup_{i<j}J_f^{\alpha_i}=J_f^{\alpha_j}$, the same for $J_g$. By the induction hypothesis and the conditions b) and c) for $i<j$, we have $dom(F_{\alpha_j})=J_f^{\alpha_j}$ (this is the condition b)) and $rang(F_{\alpha_j})=J_g^{\alpha_j}$. This and Claim 4.8 ensure that condition d) is satisfied. By the induction hypothesis, for every $i<j$, $F_{\alpha_i}$ satisfies conditions e) and f), then $F_{\alpha_j}$ satisfies conditions e) and f). \hfill $_{\square_{\text{Claim 4.10}}}$\\ \\
Define $F=\cup_{i<\kappa}F_{\alpha_i}$, clearly, it is an isomorphism between $J_f$ and $J_g$.
\end{proof}
From now on $\kappa$ will be an inaccessible cardinal. Let us take a look to the sets $rang(f)$ and $rang(c_f)$, more specific to the set $\{\alpha<\kappa|f(\alpha)\in rang(c_f)\}$.\\ \\
{\bf Remark.} Assume $f\in \kappa^\kappa$ and let $J_f$ be the respective coloured tree obtained by Definition 4.6. If $\eta\in J_f$ satisfies Definition 4.6 item 8 b), then clearly exists $\alpha<\kappa$ such that $c_f(\eta)=f(\alpha)$. It is possible that not for every $\alpha<\kappa$, there is $\eta\in J^{\alpha+1}_f$ such that $c_f(\eta)=f(\alpha)$. Nevertheless the set $C=\{\alpha<\kappa|\exists \xi\in J^{\alpha+1}_f\text{ such that }\xi_1=id_\omega+1\text{ and } c_f(\xi)=f(\alpha)\}$ is an $\omega$-club. $C$ is unbounded: For every $\beta<\kappa$ we can construct the function $\eta\in J_f$ by $\beta_0=\beta$, $\eta_1=id_\omega+1$, $\eta_2(i)=\beta_i$, $\eta_3(i)=\beta_i+1$, $\eta_4(i)=\gamma_i$ and $\eta_5=\eta_2$, where $\gamma_i$ is the least ordinal such that $P_{\gamma_i}^{\beta_i \beta_i+1}=\{\xi:[i,i+1)\rightarrow [\beta_i,\beta_i+1)\}$ and $\beta_{i+1}=\beta_i+1+\gamma_i$; since $\kappa$ is inaccessible, $\eta\in J^{(\cup_{i<\omega}\beta_i)+1}_f$ and $\cup_{i<\omega}\beta_i\in \mathcal{C}$. $C$ is closed: Let $\{\alpha_i\}_{i<\omega}$ be a succession of elements of $\mathcal{C}$, for every $i<\omega$ let $\xi^i$ be an element of $J_f$ such that $\xi^i_1=id_\omega+1$ and $rang(\xi_5^i)=\alpha_i$, define $n_0=0$ and for every $i<\omega$, $n_{i+1}$ as the least natural number bigger than $n_i$ such that $\alpha_i<\xi^{i+1}_2(n_{i+1})$. The function $\xi$ define by $\xi\restriction_{[n_i,n_{i+1})}=\xi^i\restriction_{[n_i,n_{i+1})}$ is an element of $J^{(\cup_{i<\omega}\alpha_i)+1}_f$ such that $\xi_1=id_\omega+1$ and $rang(\xi_5)=\cup_{i<\omega}\alpha_i$, therefore $f(\cup_{i<\omega}\alpha_i)=c_f(\xi)$ and $\cup_{i<\omega}\alpha_i\in \mathcal{C}$.


\section{The Orthogonal Chain Property}
In this section we will construct a model of $T$ from an element of $\kappa^\kappa$. Before this, let us fix some notation and make some general assumptions. 
From now on $T$ is going to be a stable theory. Denote by $\lambda(T)$ the least cardinal such that $T$ is $\lambda$-stable, $\lambda_r(T)$ the least regular cardinal $\lambda$ bigger or equal than $\lambda(T)$. And $\kappa$ will be bigger than $\lambda_r(T)$. \\ \\
For every $J\subseteq \kappa^{\leq\omega}$ closed under initial segments, order $I=\mathcal{P}_\omega(J)$ by $\leq$ as, for every $u,v\in I$ we say $u\leq v$ if for every $\eta\in u$ exists $\xi\in v$ such that $\eta$ is an initial segment of $\xi$. Let us denote by $r(\eta,\xi)$ the longest element in $J$ that is an initial segment of both, and $u\cap^*v$ the largest set that satisfies:
\begin{itemize}
\item $u\cap^*v\subseteq \{r(\eta,\xi)|\eta\in u,\xi\in v\}$
\item if $\tau \in u\cap^*v$, $\eta\in u$, $\xi\in v$ and $\tau$ is an initial segment of $r(\eta, \xi)$ then $\tau=r(\eta,\xi)$
\end{itemize}
\begin{definicion}
Assume $J\subseteq \kappa^{\leq\omega}$ is closed under initial segments and $I=\mathcal{P}_\omega(J)$. We say that an indexed family $\Sigma=\{A_u|u\in I\}$ is strongly independent if:
\begin{itemize}
\item For every $u,v\in I$, $u\leq v$ implies $A_u\subseteq A_v$.
\item if $u,u_i\in I$ for $i<n$ and $B\subseteq \cup_{i<n}A_{u_i}$ has power less than $\lambda_r(T)$, then there is an automorphism of the monster model $f=f^{\Sigma, B}_{u,u_0,\ldots,u_{n-1}}$, such that $f\restriction_{(B\cap A_u)}=id_{B\cap A_u}$ and $f(B\cap A_{u_i})\subseteq A_{u\cap^*u_i}$.
\end{itemize}
\end{definicion}
We will construct models using an isolation notion. In \cite{HS98} Shelah gives an axiomatic approach for isolation notion and defines $F$-constructible, $F$-primary and $F$-atomic where $F$ is an isolation notion.
\begin{definicion}
Denote by $F_{\lambda_r(T)}^s$ the set of pairs $(p,A)$ with $|A|<\lambda_r(T)$, such that for some $B\supseteq A$, $p\in S(B)$, and $p\restriction_A\vdash p$.
\end{definicion}
$F_{\lambda_r(T)}^s$ is the isolation notion we are going to use. Instead of write $F_{\lambda_r(T)}^s$-constructible, $F_{\lambda_r(T)}^s$-primary and $F_{\lambda_r(T)}^s$-atomic we will write $s$-constructible, $s$-primary and $s$-atomic.\\
Now we can state in detail the lemma that leads us to the construction of $\mathcal{A}^f$ from the coloured tree $J_f$. The proof of this lemma can be found in \cite{HS98} (Theorem 4 and Claim (I)).
\begin{lema}
Assume that $\Sigma=\{A_u|u\in I\}$, $I=\mathcal{P}_\omega (J)$ is strongly independent. Then there are subsets of the monster model, $\mathcal{A}_u$ for $u\in I$, such that
\begin{itemize}
\item [(a)] for all $u,v\in I$, $u\leq v$ implies $\mathcal{A}_u\subseteq \mathcal{A}_v$
\item [(b)] for all $u\in I$, $\mathcal{A}_u$ is $s$-primary over $A_u$; in fact it is $s$-primary over $\cup_{v<u}\mathcal{A}_v$ (see the proof of Theorem 4 in \cite{HS98})
\item [(c)] $\cup_{u\in I}\mathcal{A}_u$ is a model
\item [(d)] if $v\leq u$, then $\mathcal{A}_u$ is $s$-atomic over $\cup_{\eta\in J_f}A_\eta$ and $s$-primary over $\mathcal{A}_v\cup A_u$; in fact for all $a\in \mathcal{A}_u$ there is $B\subset A_u$ of power less than $\lambda_r(T)$, such that $t(a,B)\vdash t(a,\cup_{\eta\in J_f}A_\eta)$ (see the proof of Theorem 4 in \cite{HS98})
\item [(e)] if $J'\subseteq J$ is closed under initial segments and $u\in \mathcal{P}_\omega(J')$, then $\cup_{v\in \mathcal{P}_\omega(J')}\mathcal{A}_v$ is $s$-constructible over $\mathcal{A}_u\cup\bigcup_{v\in \mathcal{P}_\omega(J')}A_v $
\item [(f)] the family $\{\mathcal{A}_u|u\in I\}$ is strongly independent (see Claim (I) in the proof of Theorem 4 in \cite{HS98})
\end{itemize}
\end{lema}
\noindent
In \cite{HS98} the models for Lemma 5.3 above, are constructed as follow:
Let $\{u_i|i<\beta\}$ be an enumeration of $I$ such that $u_i\leq u_j$ and $u_j\nleq u_i$ implies $i\leq j$. Choose $\alpha$, $\gamma_i<\alpha$ for $i<\beta$, $a_\gamma$ and $B_\gamma$ for $\gamma<\alpha$, and $s:\alpha\rightarrow I$ so that
\begin{enumerate}
\item $\gamma_0=0$ and $(\gamma_i)_{i<\beta}$ is increasing and continuous,
\item if $\gamma_i\leq\gamma<\gamma_{i+1}$, then $s(\gamma)=u_i$,
\item for all $\gamma<\alpha$, $|B_\gamma|<\lambda$ and if we write for $\gamma\leq \alpha$, $A_u^\gamma=A_u\cup\{a_\delta|\delta<\gamma,s(\delta)\leq u\}$, then $B_\gamma\subseteq A_{s(\gamma)}^\gamma$,
\item for all $\gamma<\alpha$, if we write $A^\gamma=\cup_{u\in I}A^\gamma_u$, then $t(a_\gamma,B_\gamma)$ $s$-isolates $t(a_\gamma,A^\gamma)$,
\item for all $i<\beta$, there are no $a\notin A_{u_i}^{\gamma_{i+1}}$ and $B\subseteq A^{\gamma_{i+1}}_{u_i}$ of power less than $\lambda$ such that $t(a,B)$ $s$-isolates $t(a,A^{\gamma_{i+1}})$,
\item if $a_\delta\in B_\gamma$, then $B_\delta\subseteq B_\gamma$.
\end{enumerate}
For all $u\in I$, $\mathcal{A}_u=A_u^\alpha$.\\ \\
By 3 and 4, $\mathcal{A}_u$ is s-constructible over $\cup_{v<u}\mathcal{A}_v$.\\ \\
At this point it is clear that our intention is to use Lemma 5.3 with $I=\mathcal{P}_\omega(J_f)$. We only need to find the appropriate sets $A_u$ for us. We will use the orthogonal chain property to construct a strongly independent family $\Sigma=\{A_u|u\in I\}$ with some properties useful for us.
The orthogonal chain property implies that $T$ is unsuperstable, as we will see later. 
\begin{definicion}
$T$ has the orthogonal chain property (OCP), if there exist $\lambda_r(T)$-saturated models of $T$ of power $\lambda_r(T)$, $\{\mathcal{A}_i\}_{i<\omega}$, $a\notin \cup_{i<\omega}\mathcal{A}_i$, such that $t(a,\cup_{i<\omega}\mathcal{A}_i)$ is not algebraic for every $j<\omega$, $t(a,\cup_{i<\omega}\mathcal{A}_i)\perp \mathcal{A}_j$, and for every $i\leq j$, $\mathcal{A}_i\subseteq \mathcal{A}_j$.
\end{definicion}
The OCP is similar to the DIDIP defined by Shelah in \cite{Sh}.\\
If $T$ has the OCP then $T$ is unsuperstable, the chain $\mathcal{A}_i\subseteq \mathcal{A}_j$ and $a$ satisfy $a\not\downarrow_{\mathcal{A}_i}\mathcal{A}_{i+1}$.\\
To show this, assume $T$ is superstable and has the OCP. Let $\{\mathcal{A}_i\}_{i<\omega}$ be the chain given by the OCP and construct the following chain by induction. Let $B_0$ and $B_1$ be the least elements of $\{\mathcal{A}_i\}_{i<\omega}$ such that $B_0\subset B_1$ and $a\not\downarrow_{B_0}B_1$. For every $0<i<\omega$ let $B_{i+1}$ be the least element of $\{\mathcal{A}_i\}_{i<\omega}$ that satisfies $B_i\subset B_{i+1}$ and $a\not\downarrow_{B_i}B_{i+1}$. Since $T$ is superstable, this chain is finite, let $B_n$ be the biggest element of this chain. By the inductive construction of $\{B_i\}_{i\leq n}$ we know that $a\downarrow_{B_n}\mathcal{A}_j$ for every $B_n\subset\mathcal{A}_j$. Therefore, for every finite subset $A\subset \cup_{i<\omega}\mathcal{A}_i$, $a\downarrow_{B_n}A$ and by the finite character $a\downarrow_{B_n}\cup_{i<\omega}\mathcal{A}_i$. By assumption $T$ has the OCP, then $t(a,\cup_{i<\omega}\mathcal{A}_i)\perp \mathcal{A}_j$ for every $j$, in particular $t(a,\cup_{i<\omega}\mathcal{A}_i)\perp B_n$. So $a\downarrow_{\cup_{i<\omega}\mathcal{A}_i}a$ and $t(a,\cup_{i<\omega}\mathcal{A}_i)$ is algebraic, a contradiction.\\
From now on we will assume that $T$ has the OCP.\\ \\
The following is the construction of the family $\Sigma=\{A^f_u|u\in I\}$ from $J_f$ using the OCP. By Definition 4.6 $J_f\subseteq (\omega\times\kappa^4)^{\leq \omega}$, we will denote by $\mathcal{X}$ the set $\omega\times\kappa^4$.\\
Let $a$ and $\{\mathcal{A}_i^f\}_{i<\omega}$ be the ones witnessing the OCP for $T$.
Since for every saturated model, $B\supset \mathcal{A}$ and $C$, there is $D$ such that $t(C,\mathcal{A})=t(D,\mathcal{A})$ and $D\downarrow_\mathcal{A}B$. Then  we can find for each $\eta\in (J_f)_\omega$ ($(J_f)_\omega=\{x\in J_f|ht(x)=\omega\}$) automorphisms of the monster model, $\{H_{\eta\restriction_i}\}_{i\leq \omega}$ and models $\{\mathcal{A}_{\eta\restriction_i}^f\}_{i\leq\omega}$ , that satisfies 
\begin{itemize}
\item $H_\eta(\mathcal{A}_i^f)=\mathcal{A}^f_{\eta\restriction_i}$. 
\item $H_{\eta\restriction_i}=H_\eta\restriction_{\mathcal{A}^f_i}$. 
\item Define $A^f_u$ for each $u\subseteq J_f$ as $A^f_u=\cup_{\eta\in u}\mathcal{A}^f_\eta$.
\item Define $U(\xi,\alpha)$ for every $\xi\in J_f\cap \mathcal{X}^{<\omega}$ and $\alpha\in \mathcal{X}$, as the set of the $\zeta\in J_f$ that extend $\xi ^\frown \langle dom(\xi),\alpha\rangle$, and $V(\xi,\alpha)=J_f\backslash U(\xi,\alpha)$. Then $$A^f_{U(\xi,\alpha)} \downarrow_{\mathcal{A}^f_\xi}A^f_{V(\xi,\alpha)}.$$
\item $\mathcal{A}^f_\eta$ is the $s$-primary model over $\bigcup_{i<\omega}\mathcal{A}^f_{\eta\restriction i}\cup \bigcup_{i<c_f(\eta)}\{a_i\}$ where $\{a_i\}_{i<c_f(\eta)}$ is an independent sequence of elements satisfying the type $t(H_\eta(a),\mathbb{A}(\eta))$, $\mathbb{A}(\eta)=\bigcup_{i<\omega}\mathcal{A}^f_{\eta\restriction i}$.
\end{itemize}
This construction was made in \cite{hy} and \cite{HS98}. In \cite{HS98} is proven that the family $\{A^f_u|u\in \mathcal{P}_\omega(J_f)\}$, is strongly independent.\\ \\
{\bf Remark.} Notice that for every $\eta\in (J_f)_\omega$, $\mathcal{A}^f_\eta$ is $s$-primary over $\bigcup_{i<\omega}\mathcal{A}^f_{\eta\restriction i}\cup \bigcup_{i<c_f(\eta)}\{a_i\}$ and since $T$ is countable, then $$|\mathcal{A}^f_\eta|\leq \lambda(T)+(|\bigcup_{i<\omega}\mathcal{A}^f_{\eta\restriction i}\cup \bigcup_{i<c_f(\eta)}\{a_i\}|+\lambda_r(T))^\omega.$$ If $f$ satisfies, $|f(\alpha)^\omega|=|f(\alpha)|$, $\lambda_r(T)<f(\alpha)$, and $f(\alpha) =c_f(\eta)$, for some $\alpha<\kappa$ and $\eta\in J_f$. Then $|\mathcal{A}^f_\eta|=c_f(\eta)$.\\ \\
For every $\eta\in (J_f)_\omega$ denote by $p_\eta^f$ the type $t(H_\eta(a),\mathbb{A}(\eta))$ clearly $p_\eta^f\perp \mathcal{A}^f_{\eta \restriction i}$ for every $i<\omega$. Denote by $\mathcal{I}_f$ the set $\mathcal{P}_\omega(J_f)$, the family $\{A^f_u|u\in \mathcal{I}_f\}$ is strongly independent, by Lemma 5.3 we obtain the models $\{\mathcal{A}^f_u\}_{u\in \mathcal{I}_f}$ and $\mathcal{A}^f=\cup_{u\in \mathcal{I}_f}\mathcal{A}^f_u$.\\
We will write $\mathcal{A}_\eta$ and $A_\eta$ instead of $\mathcal{A}^f_\eta$ and $A^f_\eta$, when it is clear and there is no possibility of ambiguity.\\ 
Under some assumptions on $f$ and $g$, elements of $\kappa^\kappa$, the models $\mathcal{A}^f$ and $\mathcal{A}^g$ are isomorphic if and only if $f$ and $g$ are $E^\kappa_{\omega\text{-club}}$ related. The proof of this is made by a dimension argument. Therefore, before we start with the proof we need to do some calculation, like calculate $|\mathcal{A}_f^{<\alpha}|$ and others.
\begin{fact}
Let $I_f^\alpha=\mathcal{P}_\omega(J_f^\alpha)$, where $J_f^\alpha=\{\eta\in J_f| rang(\eta)\subset \omega\times(\beta)^4\text{ for some }\beta<\alpha\}$ (as in the proof of Lemma 4.7), define $\mathcal{A}_f^{\alpha}=\cup_{u\in I_f^\alpha }\mathcal{A}_u$. If for every $\beta<\kappa$, $\lambda_r(T)<f(\beta)$, $|f(\beta)^\omega|=|f(\beta)|$ and $\beta<f(\beta)$, then there exists a club such that every $\alpha$ in that club satisfies $|\mathcal{A}_f^{\alpha+1}|\leq sup(\{f(\beta)\}_{\beta\leq\alpha})$.
\end{fact}
\begin{proof}
Let $C$ be the club $\{\alpha<\kappa|\forall\gamma<\alpha(\gamma^\omega<\alpha\text{ and } sup(\{c_f(\eta)\}_{\eta\in J_f^\gamma})<\alpha)\}$. Assume $u$ is such that there is at least one $\xi\in u$ such that $f(\beta)=c_f(\xi)$ for some $\beta$, then by the previous remark $|A_u|=|\cup_{\eta\in u}\mathcal{A}_\eta|=max(\{c_f(\eta)^\omega\}_{\eta\in u})$. Since $\mathcal{A}_u$ is $s$-primary over $A_u$ we get \\ $|\mathcal{A}_u|\leq \lambda(T)+(|A_u|+\lambda_r(T))^\omega=max(\{c_f(\eta)^\omega\}_{\eta\in u})$. Therefore for every $\alpha<\kappa$ $$|\mathcal{A}_f^{\alpha+1}|\leq |J_f^{\alpha+1}|\cdot sup(\{(c_f(\eta)^\omega)^\omega\}_{\eta\in J_f^{\alpha+1}}),$$ if $\alpha\in C$ then $|J_f^{\alpha+1}|=\cup_{\beta\leq \alpha}\beta^\omega\leq\alpha^\omega\leq f(\alpha)^\omega=f(\alpha)$, so $$|\mathcal{A}_f^{\alpha+1}|\leq f(\alpha)\cdot sup(\{(c_f(\eta)^\omega)^\omega\}_{\eta\in J_f^{\alpha+1}}).$$ Also for every $\eta\in J_f^\alpha$, $c_f(\eta)<f(\beta)$ for some $\beta<\alpha$, therefore $$|\mathcal{A}_f^{\alpha+1}|\leq sup(\{f(\beta)\}_{\beta\leq\alpha},\{(c_f(\eta)^\omega)^\omega\}_{\eta\in J_f^{\alpha+1}\backslash J_f^{\alpha}}).$$ But every $\eta\in J_f^{\alpha+1}\backslash J_f^{\alpha}$ with $dom(\eta)=\omega$ has $rang(\eta_1)=\omega$ and $f(\alpha)=c_f(\eta)$, otherwise $rang(\eta_5)<\alpha$ and $\eta\in J_f^{\alpha}$. We conclude $|\mathcal{A}_f^{\alpha+1}|\leq sup(\{f(\beta)\}_{\beta\leq\alpha})$.
\end{proof}
\begin{fact}
If $w\leq u$, then $A_u \rhd_{\mathcal{A}_w}\mathcal{A}_u$.
\end{fact}
\begin{proof}
If $w\leq u$, by Lemma 5.3, $\mathcal{A}_u$ is $s$-constructable over $\mathcal{A}_w\cup A_u$, and $\mathcal{A}_w$ is saturated, then $A_u \rhd_{\mathcal{A}_w}\mathcal{A}_u$. 
\end{proof}
{\bf Remark.} Notice that Fact 5.6 implies that $A_u \rhd_{\mathcal{A}_\eta}\mathcal{A}_u$ for every $\eta\in u$.
\begin{fact}
Assume $f\in \kappa^\kappa$. If $u,v\in \mathcal{I}_f$, then $\mathcal{A}_u\downarrow_{\mathcal{A}_{u\cap^*v}}\mathcal{A}_v$.
\end{fact}
\begin{proof}
By the previous fact it is enough to prove $A_u\downarrow_{\mathcal{A}_{u\cap^*v}}A_v$.\\
We will define $U(\xi,\alpha)$ and $V(\xi,\alpha)$ for every $\xi\in u\cap^*v$ and $\alpha\in \mathcal{X}$ such that $\xi^\frown \langle dom(\xi),\alpha\rangle$ is extended by an element of $u$ and non of the elements of $v$ extend it. Denote by $U(\xi,\alpha)$ the set of those $\eta\in J_f$ that extends $\xi^\frown \langle dom(\xi),\alpha\rangle$ and $V(\xi,\alpha)=J_f\backslash U(\xi,\alpha)$.\\
By the construction of the sets $\mathcal{A}_u$, for every $\xi \in u\cap^*v$ and every $\alpha\in \mathcal{X}$ such that $U(\xi,\alpha)$ and $V(\xi,\alpha)$ are defined, we have
\begin{equation}
A_{U(\xi,\alpha)} \downarrow_{\mathcal{A}_\xi} A_{V(\xi,\alpha)}.
\end{equation}
Let $U=\cup \{U(\xi,\alpha)|\{\xi^\frown \langle dom(\xi),\alpha\rangle\}\leq u, \xi\in u\cap^*v, \alpha\in \mathcal{X}\}$ and $V$ be $J_f\backslash U$. Then by transitivity and (4),
$$ A_U \downarrow_{A_{u\cap^*v}} A_V.$$ 
By the way $U$ and $V$ were defined, $u\subseteq U$ and $v\subseteq V$. Therefore $$A_u\downarrow_{A_{u\cap^*v}}A_v.$$ By the in fact part of Lemma 5.3 (d), $\mathcal{A}_{u\cap^*v}\downarrow_{A_{u\cap^*v}}\cup_{\xi\in J_f}\mathcal{A}_\xi$. Therefore $\mathcal{A}_{u\cap^*v}\downarrow_{A_{u\cap^*v}}A_uA_v$ and we conclude $A_u\downarrow_{\mathcal{A}_{u\cap^*v}}A_v$.
\end{proof}
\begin{lema}
Assume $f\in \kappa^\kappa$ is such that for every $\alpha$, $f(\alpha)>\lambda_r(T)$, $f(\alpha)^\omega=f(\alpha)$, and $rang(f)\subset Card$. If $\eta\in (J_f)_\omega$ is such that $c_f(\eta)=f(\alpha)$, then $dim(p_\eta^f, \mathcal{A}^f)=c_f(\eta)$.
\end{lema}
\begin{proof}
Suppose not. Then there exists an independent sequence $I\subseteq \mathcal{A}^f$ over $\mathbb{A}(\eta)$ such that $|I|>c_f(\eta)$ and $a\models p_\eta^f$ for every $a\in I$. By a previous remark we know that $c_f(\eta)=|\mathcal{A}_\eta|$, so there exists $b\in I\backslash \mathcal{A}_\eta$ such that $b\downarrow_{\mathbb{A}(\eta)}\mathcal{A}_\eta$. Thus $t(b,\mathcal{A}_\eta)\perp \mathcal{A}_{\eta\restriction_i}$ for all $i<\omega$.\\
There exists $u\in \mathcal{I}_f$ such that $\eta\in u$ and $b\in \mathcal{A}_u$. By Fact 5.7 we know that there exists $i<\omega$ such that $\mathcal{A}_{u\backslash \{\eta\}} \downarrow_{\mathcal{A}_{\eta\restriction_i}}\mathcal{A}_\eta$.\\
Since $t(b,\mathcal{A}_\eta)\perp \mathcal{A}_{\eta\restriction_i}$, $b\downarrow_{\mathcal{A}_\eta}A_{u\backslash \{\eta\}}$. So $b\downarrow_{\mathcal{A}_\eta}A_u$ and by a previous remark we know that $A_u \rhd_{\mathcal{A}_\eta}\mathcal{A}_u$, thus $b\downarrow_{\mathcal{A}_\eta}\mathcal{A}_u$. But $b\in \mathcal{A}_u$, so $t(b,\mathcal{A}_\eta)$ is algebraic. By the choice of $b$, $t(b,\mathcal{A}_\eta)$ is a non-forking extension of $p_\eta^f$. This implies that $p_\eta^f$ is algebraic. By the OCP, $p_\eta^f$ is not algebraic; a contradiction.
\end{proof}
The Theorem 5.9 gives a reduction only for certain elements of $\kappa^\kappa$, as we will see in Corollary 5.10, this can be easy generalize to all the elements of $\kappa^\kappa$.
\begin{teorema}
Assume $f,g$ are functions from $\kappa$ to $Card\backslash \lambda_r(T)$, that satisfy for every $\beta<\kappa$, $f(\beta)^\omega=f(\beta)$, $g(\beta)^\omega=g(\beta)$ and for every cardinal $\alpha$, $f(\alpha)>\alpha^{++}$, $g(\alpha)>\alpha^{++}$. Then the models $\mathcal{A}^f$ and $\mathcal{A}^g$ are isomorphic if and only if $f$ and $g$ are $E^\kappa_{\omega\text{-club}}$ related.
\end{teorema}
\begin{proof}
From right to left.\\
By Lemma 4.7 if $f\ E^\kappa_{\omega\text{-club}}\ g$ then $J_f\cong J_g$. Let $G:J_f \rightarrow J_g$ be an isomorphism. \\
We will construct, using induction, a family of function $\{F_u\}_{u\in \mathcal{I}_f}$ such that $F_u:\mathcal{A}_u\rightarrow\mathcal{A}_{G[u]}$ is an isomorphism and $\cup_{v<u}F_v\subseteq F_u$. Notice that this is equivalent to: a family of function $\{F_u\}_{u\in \mathcal{I}_f}$ such that $F_u:\mathcal{A}_u\rightarrow\mathcal{A}_{G[u]}$ is an isomorphism and for every $W\subseteq\mathcal{I}_f$, $\cup_{v\in W}F_v:\cup_{v\in W}\mathcal{A}_v\rightarrow\mathcal{A}^g$ is an elementary map.\\
Let $\{u_i:i<\alpha^*\}$ be the enumeration of $\mathcal{I}_f$ used in the construction of the models $\mathcal{A}_u$ (see Lemma 5.3).\\
Our induction hypothesis for $\beta<\alpha^*$ is the following:\\
The functions $\{F_{u_i}\}_{i<\beta}$ satisfy
\begin{itemize}
\item For all $i<\beta$, $F_{u_i}:\mathcal{A}_{u_i}\rightarrow\mathcal{A}_{G[u_i]}$ is an isomorphism.
\item $\cup_{i<\beta}F_{u_i}$ is an elementary map. 
\end{itemize}
By Lemma 5.3, $\mathcal{A}_{u_\beta}$ and $\mathcal{A}_{G[u_\beta]}$ are $s$-primary over $\cup_{v<u_\beta}\mathcal{A}_v$ and $\cup_{v<u_\beta}\mathcal{A}_{G[v]}$ respectively. By the induction hypothesis $\cup_{v<u_\beta}F_v$ is elementary and onto $\cup_{v<u_\beta}\mathcal{A}_{G[v]}$. Since the $s$-primary models over $\cup_{v<u_\beta}\mathcal{A}_v$ are isomorphic and the $s$-primary models over $\cup_{v<u_\beta}\mathcal{A}_{G(v)}$ are isomorphic, there is an isomorphism from $\mathcal{A}_{u_\beta}$ to $\mathcal{A}_{G[u_\beta]}$ that extends $\cup_{v<u_\beta}F_v$. Let us define $F_{u_\beta}$ as this isomorphism.\\
We will prove that $\cup_{i\leq\beta}F_{u_i}$ is elementary by proving that for every $n<\omega$ and every sequence $x_0,x_1,\ldots x_n\in\{u_i|i\leq\beta\}$, the map $\cup_{i\leq n}F_{x_i}:\cup_{i\leq n}\mathcal{A}_{x_i}\rightarrow\mathcal{A}^g$ is elementary.\\
Clearly we can assume that $n>0$, $x_n=u_\beta$, $u_\beta$ is not comparable with $x_{n-1}$, and $u_i\neq u_j$ for every $i\neq j$.\\
Define $u'=\cup_{i<n}(x_n\cap^* x_i)$, notice that $u'\leq u_\beta$.\\ \\
{\bf Case $u'< u_\beta$}\\ 
Let $X=\cup_{i<n}x_i$, by Fact 5.7 $\mathcal{A}_{u_\beta}\downarrow_{\mathcal{A}_{u'}}\mathcal{A}_X$, therefore 
\begin{equation}
\mathcal{A}_{u_\beta}\downarrow_{\mathcal{A}_{u'}}\cup_{i<n}\mathcal{A}_{x_i}.
\end{equation} 
Since $G$ is an isomorphism, $G[u]\cap^* G[v]=G[u\cap^*v]$ for every $u,v\in \mathcal{I}_f$. By Fact 5.7 $\mathcal{A}_{G[u_\beta]}\downarrow_{\mathcal{A}_{G[u']}}\mathcal{A}_{G[X]}$, therefore 
\begin{equation}
\mathcal{A}_{G[u_\beta]}\downarrow_{\mathcal{A}_{G[u']}}\cup_{i<n}\mathcal{A}_{G[x_i]}.
\end{equation} 
By the induction hypothesis $\cup_{i<\beta}F_{u_i}$ is elementary and thus there exists an automorphism of the monster model $F$ that extends $\cup_{i<\beta}F_{u_i}$. By (6) 
\begin{equation}
F^{-1}[\mathcal{A}_{G[u_\beta]}]\downarrow_{\mathcal{A}_{u'}}\cup_{i<n}\mathcal{A}_{x_i}.
\end{equation} 
Since $F$ and $F_{u_\beta}$ both extend $F_{u'}$ we conclude $t(\mathcal{A}_{u_\beta},\mathcal{A}_{u'})=t(F^{-1}[\mathcal{A}_{G[u_\beta]}],\mathcal{A}_{u'})$ and it is a stationary type. So by (5) and (7), the types $t(\mathcal{A}_{u_\beta}, \cup_{i<n}\mathcal{A}_{x_i})$ and $t(F^{-1}[\mathcal{A}_{G[u_\beta]}], \cup_{i<n}\mathcal{A}_{x_i})$ are equal, therefore
$$t(\mathcal{A}_{u_\beta}\ ^\frown \cup_{i<n}\mathcal{A}_{x_i}, \emptyset)=t(\mathcal{A}_{G[u_\beta]}\ ^\frown \cup_{i<n}F_{x_i}[\mathcal{A}_{x_i}], \emptyset).$$ Therefore $\cup_{i\leq n}F_{x_i}$ is elementary.\\ \\
{\bf Case $u'= u_\beta$}\\
Let $(a_0,a_1,\ldots , a_n)$ be any tuple such that for all $i\leq n$, $a_i\in \mathcal{A}_{x_i}$. Define $\mathcal{A}'=\cup_{v<u_\beta}\mathcal{A}_v$ and $$F'=\bigcup_{v<u_\beta\text{ or }v\in\{x_i|i<n\}}F_v$$ By the induction hypothesis $F'$ is elementary and by Lemma 5.3 $\mathcal{A}_{u_\beta}$ is $s$-constructible over $\mathcal{A}'$, therefore $\mathcal{A}_{u_\beta}$ is $s$-atomic over $\mathcal{A}'$. Then there is $A'\subseteq \mathcal{A}'$ of size less than $\lambda_r(T)$ such that $t(a_n, A')\vdash t(a_n,\mathcal{A}')$. By Lemma 5.3 $\{\mathcal{A}_u|u\in \mathcal{I}_f\}$ is a strongly independent family. Let $B=A'\cup\{a_i|{i<n}\}$, there is an automorphism of the monster model $H$, that satisfies $H\restriction_{A'}=id$ (notice that $A'=B\cap\mathcal{A}_{u_\beta}$) and $H(a_i)\in \mathcal{A}_{x_n\cap^*x_i}$ for every $i<n$, therefore $H(a_i)\in\mathcal{A}'$. Since $$t(a_n, A')\vdash t(a_n,\mathcal{A}')$$ and $$t(F_{u_\beta}(a_n),F_{u_\beta}(A'))\vdash t(F_{u_\beta}(a_n),F'(\mathcal{A}'))$$ so $$t(F_{u_\beta}(a_n),F_{u_\beta}(A'))\vdash t(F_{u_\beta}(a_n),\cup_{i<n}F_{x_i}(a_i)).$$ We conclude that $\cup_{i\leq n}F_{x_i}$ is elementary.\\ \\
We conclude that $\mathcal{A}^f\cong\mathcal{A}^g$.\\ \\
From left to right.\\
Let us assume that $f$ and $g$ are not $E^\kappa_{\omega\text{-club}}$ related but there is an isomorphism $\Pi:\mathcal{A}^g\rightarrow\mathcal{A}^f$. \\
By a previous remark we know that $\{\alpha<\kappa|\exists \eta\in J^{\alpha+1}_g (c_g(\eta)=g(\alpha))\}$ contains an $\omega$-club and by Fact 5.5 there is a club such that for every $\alpha$ in it, $|\mathcal{A}_g^{\alpha+1}|\leq sup(\{g(\beta)\}_{\beta\leq \alpha})$ (this also holds for $f$). Therefore, there is an $\omega$-club such that every element of it, $\alpha$, satisfies
\begin{itemize}
\item $|\mathcal{A}_f^{\alpha+1}|\leq sup(\{f(\beta)\}_{\beta\leq \alpha})$.
\item $|\mathcal{A}_g^{\alpha+1}|\leq sup(\{g(\beta)\}_{\beta\leq \alpha})$.
\item There exists $\eta\in J^{\alpha+1}_g$ and $\xi\in J^{\alpha+1}_f$ that satisfy $sup(rang(\eta_5))=sup(rang(\xi_5))=\alpha$, $c_g(\eta)=g(\alpha)$, and $c_f(\xi)=f(\alpha)$.
\end{itemize}
Since $\{\alpha<\kappa|\forall\beta<\alpha (f(\beta),g(\beta)<\alpha)\}$ and $\{\alpha<\kappa|\Pi(\mathcal{A}_g^{\alpha})= \mathcal{A}_f^{\alpha}|\}$ are clubs, and $f$ and $g$ are not $E^\kappa_{\omega\text{-club}}$ related (the set $\{\alpha<\kappa|f(\alpha)\neq g(\alpha)\}$ intersect every $\omega$-club), we can assume the existence of an ordinal $\alpha$ with countable cofinality such that:
\begin{itemize}
\item For every $\beta<\alpha$, $f(\beta)<\alpha$ and $g(\beta)<\alpha$.
\item $g(\alpha)\neq f(\alpha)$.
\item There exists $\eta\in J^{\alpha+1}_g$ such that $c_g(\eta)=g(\alpha)$.
\item There exists $\xi\in J^{\alpha+1}_f$ such that $c_f(\xi)=f(\alpha)$.
\item $|\mathcal{A}_f^{\alpha+1}|\leq sup(\{f(\beta)\}_{\beta\leq \alpha})$.
\item $|\mathcal{A}_g^{\alpha+1}|\leq sup(\{g(\beta)\}_{\beta\leq \alpha})$.
\item $\Pi(\mathcal{A}_g^{\alpha})= \mathcal{A}_f^{\alpha}$.
\end{itemize}
By symmetry we may assume that $g(\alpha)>f(\alpha)$.\\
By Lemma 5.8, $\eta$ and $\alpha$ satisfy $dim(p_\eta^g,\mathcal{A}^g)=c_g(\eta)=g(\alpha)$, so the type $\Pi(p_\eta^g)=\{\varphi(x,\Pi(c))|\varphi(x,c)\in p_\eta^g\}$ has $dim(\Pi(p_\eta^g),\mathcal{A}^f)=g(\alpha)$.\\
Since $\eta\in J_g^{\alpha+1}$ and $\Pi(\mathcal{A}_g^{\alpha})= \mathcal{A}_f^{\alpha}$, $\Pi(\mathbb{A}(\eta))\subseteq \mathcal{A}_f^{\alpha}$.
On the other hand, by the way we chose $\alpha$, we conclude that $|\mathcal{A}_f^{\alpha+1}|<g(\alpha)=dim(\Pi(p_\eta^g),\mathcal{A}^f)$. So there exists an independence sequence $A\subseteq \mathcal{A}^f$ over $\Pi(\mathbb{A}(\eta))$, such that $a\models \Pi(p_\eta^g)$, with an element $b\in A\backslash \mathcal{A}_f^{\alpha+1}$ that satisfy $b\downarrow_{\Pi(\mathbb{A}(\eta))}\mathcal{A}_f^{\alpha+1}$.\\
For every $u\in \mathcal{I}_f$ denote by $\bar{u}$ the closure of $u$ under initial segments.\\
Let $\{u_i\}_{i<g(\alpha)^+}$ be a sequence of elements of $\mathcal{I}_f$ with the following properties:
\begin{itemize}
\item $b\in \mathcal{A}_{u_0}$.
\item Every $\bar{u}_i$ is a tree isomorphic to $\bar{u}_0$.
\item If $i\neq j$, then $\bar{u}_i\cap \bar{u}_j=\bar{u}_0\cap J_f^{\alpha+1}$.
\item Every $\xi\in dom(c_f)\cap \bar{u}_0$ satisfies $c_f(\xi)= c_f(G_i(\xi))$, where $G_i$ is the isomorphism between $\bar{u}_0$ and $\bar{u}_i$.
\end{itemize}
For every $\xi\in \bar{u}_0$ such that $\xi\restriction_n\in J_f^{\alpha+1}$ and $\xi\restriction_{n+1}\in\bar{u}_0\backslash J_f^{\alpha+1}$ it holds that, by Definition 4.6 $\xi\restriction_n$ has $\kappa$ many immediate successors in $J_f\backslash J_f^{\alpha+1}$. Also by Definition 4.6 the elements of $J_f$ are all the functions $\eta:s\rightarrow \omega\times \kappa^4$ that satisfy the items 1 to 8, therefore each of this immediate successors of $\xi\restriction_n$, $\zeta$, satisfies that in the set $\{\eta\in J_f|\zeta\leq \eta\}$ there is a subtree isomorphic (as coloured tree) to $\bar{u}_0\backslash J_f^{\alpha+1}$.\\
This and the fact that $u_0$ is finite, gives the existence of the sequence $\{u_i\}_{i<g(\alpha)^+}$.\\
By the way we chose the sequence $\{u_i\}_{i<g(\alpha)^+}$, for every $i<g(\alpha)^+$, the isomorphism $G_i$ induces an isomorphism $H_i:J_f^{\alpha+1}\cup \bar{u}_0\rightarrow J_f^{\alpha+1}\cup \bar{u}_i$ such that $H_i\restriction_{J_f^{\alpha+1}}= id$. The other direction of this theorem implies that the models $\mathcal{A}(0)=\cup\{\mathcal{A}_v|v\in \mathcal{P}_\omega(J_f^{\alpha+1}\cup \bar{u}_0)\}$ and $\mathcal{A}(i)=\cup\{\mathcal{A}_v|v\in \mathcal{P}_\omega(J_f^{\alpha+1}\cup \bar{u}_i)\}$ are isomorphic and there is an isomorphism $h_i:\mathcal{A}(0)\rightarrow \mathcal{A}(i)$ such that $h_i\restriction_{\mathcal{A}_f^{\alpha+1}}=id$. Let $b_0=b$ and $b_i=h_i(b)$, for every $i<g(\alpha)^+$, then $t(b_i,\mathcal{A}_f^{\alpha+1})=t(b,\mathcal{A}_f^{\alpha+1})$. By the way $(\bar{u}_i)_{i<g(\alpha)^+}$ was constructed, Lemma 5.3 and the finite character of forking, the models $(\mathcal{A}(i))_{i<g(\alpha)^+}$ are independent over $\mathcal{A}_f^{\alpha+1}$, and thus for every $i<g(\alpha)^+$, $b_i\downarrow_{\mathcal{A}_f^{\alpha+1}}\cup_{j\neq i}b_j$. Since $b\downarrow_{\Pi(\mathbb{A}(\eta))}\mathcal{A}_f^{\alpha+1}$, then for every $i<g(\alpha)^+$, $b_i\downarrow_{\Pi(\mathbb{A}(\eta))}\mathcal{A}_f^{\alpha+1}$, so $b_i\downarrow_{\Pi(\mathbb{A}(\eta))}\cup_{j\neq i}b_j$. Therefore $\{b_i\}_{i<g(\alpha)^+}$ is an independence sequence over $\Pi(\mathbb{A}(\eta))$. We conclude that $dim (\Pi (p_\eta^g), \mathcal{A}^f)\ge g(\alpha)^+$ a contradiction with $dim (\Pi (p_\eta^g), \mathcal{A}^f)= dim (p_\eta^g, \mathcal{A}^g) =g(\alpha)$.
\end{proof}
\begin{corolario}
Assume $T$ is stable and has the OCP, then $E^\kappa_{\omega\text{-club}} \leq_c \cong_T$.
\end{corolario}
\begin{proof}
Let $f$ and $g$ be elements of $\kappa^\kappa$.
First we will construct a function $F:\kappa^\kappa\rightarrow\kappa^\kappa$ such that $f\ E^\kappa_{\omega\text{-club}}\ g$ if and only if $\mathcal{A}^{F(f)}$ and $\mathcal{A}^{F(g)}$ are isomorphic.\\ \\
For every cradinal $\alpha<\kappa$, define $S_\alpha=\{\beta<\kappa|\lambda_r(T),\alpha^{+++}<\beta\text{ and } \alpha^\omega=\alpha\}$.
Let $\mathcal{G}_\beta$ be a bijection from $\kappa$ into $S_\beta$, for every $\beta<\kappa$. For every $f\in \kappa^\kappa$ define $F(f)$ by $F(f)(\beta)=\mathcal{G}_\beta(f(\beta))$, for every $\beta<\kappa$. Clearly $f\ E^\kappa_{\omega\text{-club}}\ g$ if and only if $F(f)\  E^\kappa_{\omega\text{-club}}\ F(g)$ i.e. $\mathcal{A}^{F(f)}$ and $\mathcal{A}^{F(g)}$ are isomorphic and $F$ is continuous.\\ \\
Finally we need to find $\mathcal{G}:\{\mathcal{A}^{F(f)}|f\in \kappa^\kappa\}\rightarrow\kappa^\kappa$ such that $\mathcal{A}_{\mathcal{G}(\mathcal{A}^{F(f)})}\cong \mathcal{A}^{F(f)}$ and $f \mapsto \mathcal{G}(\mathcal{A}^{F(f)})$ is continuous. This can be done as in Lemma 3.2.
\end{proof}
\begin{corolario}
Assume $T_1$ is a classifiable theory and $T_2$ is a stable theory with the OCP, then $\cong_{T_1}\leq_c \cong_{T_2}$.
\end{corolario}
\begin{proof}
Follows from Theorem 2.8 and Corollary 5.10.
\end{proof}

\providecommand{\bysame}{\leavevmode\hbox to3em{\hrulefill}\thinspace}
\providecommand{\MR}{\relax\ifhmode\unskip\space\fi MR }
\providecommand{\MRhref}[2]{%
  \href{http://www.ams.org/mathscinet-getitem?mr=#1}{#2}
}
\providecommand{\href}[2]{#2}

\end{document}